\documentclass[review,hidelinks,onefignum,onetabnum]{siamart220329}

\input{abbrev}

\usepackage{lipsum}
\usepackage{amsfonts}
\usepackage{graphicx}
\usepackage{epstopdf}
\usepackage{algorithm,algpseudocode,amsmath,amssymb} 
\ifpdf
  \DeclareGraphicsExtensions{.eps,.pdf,.png,.jpg}
\else
  \DeclareGraphicsExtensions{.eps}
\fi

\newsiamremark{remark}{Remark}
\newsiamremark{hypothesis}{Hypothesis}
\crefname{hypothesis}{Hypothesis}{Hypotheses}
\newsiamthm{claim}{Claim}
\newsiamremark{assumption}{Assumption}
\crefname{assumption}{Assumption}{Assumptions}

\usepackage{siunitx}

\headers{M$^3$C\xspace for Hyperparameter Estimation}{}

\title{
A Majorization-Minimization with Monte Carlo Approach for Hyperparameter Estimation
\thanks{Submitted to the editors \today.
\funding{This work was partially supported by the National Science Foundation under grant DMS-2038118 (E. Buser), DMS-2411197 (J. Chung), and DMS-2411198 (A.K.\ Saibaba), and Department of Energy through the award DE-SC002318 (H. Diaz and A.K. Saibaba). Any opinions, findings, conclusions or recommendations expressed in this material are those of the author(s) and do not necessarily reflect the views of the National Science Foundation.}
}}

\author{Elle Buser,  Julianne Chung, Hugo D\'iaz, Arvind K.\ Saibaba}
\usepackage{amsopn}

\makeatletter
\newcommand*{\addFileDependency}[1]{
  \typeout{(#1)}
  \@addtofilelist{#1}
  \IfFileExists{#1}{}{\typeout{No file #1.}}
}
\makeatother

\usepackage{color}
\usepackage{xspace}

\newcommand{\mmmc}{M$^3$C\xspace}

\usepackage{mathtools}
\usepackage{mathrsfs}  
\ifpdf
\hypersetup{
  pdftitle={\texorpdfstring{\mmmc}{M3C} for Hyperparameter Estimation},
  pdfauthor={}
}
\fi

\usepackage{xr}
\begin{document}
\nolinenumbers
\maketitle

\begin{abstract}
We consider inverse problems with linear forward models and Gaussian priors, but with unknown hyperparameters that may arise from the model, the noise, or the specification of the prior. We model this using a hierarchical Bayes framework resulting in a posterior distribution that is non-Gaussian, in general, and challenging to sample from. Consequently, we use an empirical Bayes framework for estimating the maximum a posteriori estimate of the hyperpameters by considering the marginalized posterior distribution. However, the optimization problem is also computationally challenging due to the need for repeated evaluation of log determinants. To address this issue, we propose a Majorization-Minimization with Monte Carlo approach, which we call \mmmc, for hyperparameter estimation.  Specifically, we replace the challenging optimization problem with a sequence of simpler ones by utilizing a majorization function (or majorant) for the log-determinant term, combined with a Monte Carlo estimator to approximate the majorant. We provide theoretical results, showing that under certain assumptions, the \mmmc iterates converge with high probability to a critical point of the original cost function. A variety of numerical examples are provided from seismic tomography, super-resolution imaging, and contaminant source identification.

\end{abstract}

\begin{keywords}
hierarchical Bayesian inverse problems, hyperparameter estimation, majorization-minimization methods, Monte Carlo, stochastic average approximation
\end{keywords}

\begin{MSCcodes}
65F22, 65M32, 62C10
\end{MSCcodes}

\section{Introduction} \label{sec:into}
Large-scale inverse problems arise in many applications, in which a critical task is to estimate or reconstruct a high-dimensional inverse parameter $\bfx \in \bbR^n$, that represents the discretization of a detailed spatial or spatio-temporal function, from a few noisy and indirect measurements $\bfb \in \bbR^m$.  Obtaining a suitable reconstruction requires an appropriate choice of the prior distribution for $\bfx$, as well as a likelihood function that incorporates knowledge of the forward model and assumptions about the measurement noise. The Bayesian approach provides a systematic approach for estimating the parameters $\bfx$ from the data $\bfb$. However, one confounding challenge in the Bayesian approach is that the forward model, the noise model, and the prior distribution have unknown/uncertain parameters that must simultaneously be estimated from the data along with the inversion parameters $\bfx$. We lump these additional parameters into so-called hyperparameters $\bftheta \in \bbR^p$, and note that this is slightly non-standard usage of this term.

The hierarchical Bayesian approach provides a general framework for modeling any \textit{additional} unknown hyperparameters 
by modeling them as random variables and estimating them simultaneously with the unknown inversion parameters.  By conditioning on the observed data, the Bayesian solution to the inverse problem is given by the joint posterior distribution $\pi(\bfx,\bftheta \mid \bfb)$ (see~\eqref{eqn:posterior}).  The joint posterior distribution is non-Gaussian in general, and a full exploration of this distribution, e.g., via Markov Chain Monte Carlo methods, can be computationally infeasible for several reasons. Namely, the number of unknown parameters is very large, and evaluating the likelihood can be very expensive in practice. 

Although our approach is general, to position our work with existing work, we consider two special but important cases. 
In what follows, let $\bfe \in \bbR^m$ represent noise or error in the data.  
\paragraph{\textbf{Case I}: Distribution uncertainty} We assume that the forward model is linear and the hyperparameters are taken to be the parameters that define the prior and the noise distributions. Specifically,  consider
\begin{align}
\label{eq:case1}
\begin{split}
 \bfb = \bfA \bfx + \bfe, & \qquad  \bfe \mid \bfpsi \sim \calN(\bfzero, \bfR(\bfpsi)), \\
 \bfx \mid \bfpsi \sim \calN(\bfmu_\bfx, \bfQ_\bfx(\bfpsi)), &  \qquad \bfpsi \sim \pi_\bfpsi, 
\end{split}
\end{align}
where $\bfA \in \bbR^{m \times n}$ and $\bfmu_\bfx\in\bbR^{n}$ are fixed, and $\bfR(\bfpsi)$ and $\bfQ_\bfx(\bfpsi)\in\bbR^{n\times n}$ are symmetric and positive definite (SPD) matrices that depend on hyperparameters $\bfpsi \in \bbR^q$ with hyperprior $\pi_\bfpsi$. 
\paragraph{\textbf{Case II}: Model uncertainty} We assume that the forward model is linear but with uncertain (hyper)parameters $\bfy$, which enter the forward model nonlinearly. This is known in the literature as a separable, nonlinear inverse problem. Specifically,
\begin{align}
\label{eq:case2}
\begin{split}
\bfb = \bfA(\bfy) \bfx + \bfe, & \qquad \bfe \sim \calN(\bfzero, \bfR), \\ \bfx \sim \calN(\bfmu_\bfx, \bfQ_\bfx), & \qquad \bfy \sim \pi_\bfy,
\end{split}
\end{align}
where $\bfA(\cdot): \bbR^\ell \to \bbR^{m \times n}$ maps unknown model parameters $\bfy \in \bbR^\ell$ to a large, forward operator and $\pi_\bfy$ represents the model uncertainty hyperparameters.  We assume that $\bfmu_\bfx\in\bbR^{n}$, $\bfR \in \bbR^{m \times m}$ and $\bfQ_\bfx\in\bbR^{n\times n}$ are fixed.

In both cases, we are interested in estimating hyperparameters: $\bfpsi$ (for Case I) and $\bfy$ (for Case II).  Thus, we consider a general framework where $\bftheta = [\bfpsi\t, \bfy\t]\t$ and refer back to \cref{eq:case1} and \cref{eq:case2} for specific derivations.

In this paper, we develop computationally efficient methods for hyperparameter estimation.  Following previous works, e.g., \cite{Hall-Hooper_Saibaba_Chung_Miller_2024,chung2024efficient}, we obtain the marginal posterior distribution by marginalizing the parameters $\bfx$ from the joint posterior distribution, and we compute the marginalized maximum a posteriori (MMAP) estimate of the hyperparameters. For large-scale problems, this presents a computationally expensive optimization problem due to the need to evaluate multiple log determinants of very large matrices.  In \cite{chung2024efficient}, stochastic average approximation (SAA) for MMAP estimation was done by using a stochastic Lanczos method. Although this approach has been shown to be successful in many cases, it can be expensive in practice and has only been applied to Case I.  

\paragraph{Our Approach and Contributions} In this paper, we develop an optimization approach for solving the MMAP based on the majorization-minimization (MM) principle~\cite{lange2016mm}. This approach uses a majorization function for the log-determinant term, to obtain a majorant for the objective function involving a trace, which leads to an inner-outer optimization scheme defined by a sequence of simpler optimization problems. The specific contributions of this paper are:
\begin{enumerate}
    \item We solve the sequence of MM majorants using a Monte Carlo estimator for the trace term in the majorant. We call this approach \textit{Majorization-Minimization with Monte-Carlo}, dubbed \mmmc (Section~\ref{sub:m3c}). 
  
    \item We conduct a probabilistic analysis of MM approaches with inexactness in the optimal solution of each majorant (Section~\ref{sec:convergenceMMMC}). As the number of iterations grows, the sequence of minimizers generated by the \mmmc iterations converges with high probability to a critical point of the original cost functional. This analysis may be of independent interest beyond this paper. 
    
    \item We derive bounds on the minimal number of samples for the trace estimator at each outer iteration to ensure that the \mmmc iterates converge with high probability (Section~\ref{sec:convergenceMMMC}).
      \item We show that the SAA approach, previously applied to Case I~\cite{chung2024efficient}, also applies to Case II, with minor modifications. We   provide a computational cost analysis of the \mmmc approach and compare it with the SAA approach (Section~\ref{ssec:compconsider}).
    \item We derive bounds on the number of Monte Carlo samples and the degree of the Lanczos polynomial to guarantee a small backward error on the approximate minimizer of the SAA approach (Section~\ref{sub:SAAconvergence}).    
\end{enumerate}
We demonstrate the performance of our approach on a variety of applications from seismic tomography, super-resolution imaging, and contaminant source identification (Section~\ref{sec:NumEx}). In particular, we will show that the proposed \mmmc method can be used for hyperparameter estimation in general hierarchical inverse problems. A background section (Section~\ref{sec:background}) and conclusions (Section~\ref{sec:conclusions}) complete the paper. 
MATLAB code implementing the proposed methods will be available at \url{https://github.com/ellebuser/hyperparam_M3C}. Additional details are available in Supplementary Materials.



\paragraph{Related work}
Hierarchical Bayesian models for inverse problems represent an active and growing area of research, e.g., see \cite{calvetti2023bayesianbook,calvetti2018inverse,sanz2025hierarchical}, and references therein. Many of these works have been investigated for linear forward models (e.g., Case I), where it is assumed that $\bfx$ is sparse or has a sparse representation.  In the context of sparsity-promoting priors, there are various classes of priors, many of which rely on a conditionally Gaussian prior with variances that are distributed according to a generalized gamma hyperprior. Notably, the Iterative Alternating Sequential algorithm described in \cite{calvetti2018inverse,calvetti2019hierachical,calvetti2020sparse} can be used to efficiently compute the MAP estimator of the joint posterior distribution. Since this is essentially an alternating optimization approach, it is easy to implement and globally convergent.  
We focus on optimization here, but sampling for hierarchical models has been considered for linear Gaussian problems in \cite{fox2016fast,saibaba2019efficient} and for sparsity-promoting priors (e.g., via prior normalization)  in \cite{glaubitz2025efficient,calvetti2025subspace}.  Since our approach can be considered as an Empirical Bayes (EB) approach, we note that an EB analysis can lead to misleading results, with smaller posterior variances and narrower credible intervals than a fully Bayesian analysis \cite{reich2026bayesian}.  Nonetheless, EB can be a useful tool especially in high-dimensional settings and when uncertainty in the hyperparameters is negligible.

Separable, nonlinear inverse problems (e.g., Case II) arise in various applications, including super-resolution, motion correction, and semiblind deconvolution \cite{chung2010efficient,espanol2023variable,cornelio2014constrained}.  Most works have focused on exploiting the separable structure to solve separable nonlinear least-squares problems efficiently, e.g., using variable projection approaches \cite{golub2003separable,ruhe1980algorithms,o2013variable,chung2010efficient,espanol2024convergence,espanol2023variable}. Since the unknown model parameters $\bfy$ can be interpreted as model error or model uncertainty \cite{kennedy2001bayesian}, there are some works that consider model error in a Bayesian framework \cite{calvetti2018inverse,arridge2006approximation,calvetti2018iterative,calvetti2024computationally,Kaipio:1338003}. Many of these works consider model error uncertainty in the context of computer model calibration or other reduced or approximate models, rather than the separable, nonlinear models of interest.  For edge-preserving tomography with uncertain view angles, which can be represented as a separable, nonlinear inverse problem, new optimization strategies are considered in \cite{riis2021computed,riis2021computed2}, and a hybrid Gibbs sampler is considered \cite{uribe2022hybrid}.

\section{Background}
\label{sec:background}
In this section, we provide the mathematical foundation for hyperparameter estimation in hierarchical Bayesian inverse problems. We first formalize the marginal posterior distribution and identify the primary computational bottlenecks. We then introduce stochastic trace estimators to approximate the marginal posterior and define the resulting empirical objective function that was optimized using an SAA method in \cite{chung2024efficient}.

\subsection{Hierarchical Bayesian Inverse Problems}
\label{sec:hierarchicalBayesian}

In this section, we describe the general marginalization approach for hyperparameter estimation in hierarchical inverse problems, and we describe specific challenges that arise in the two special cases identified in \Cref{sec:into}. For a unified presentation, let $\bftheta = [\bfpsi\t,\bfy\t]\t\in\bbR^{p}$ be the vector of hyperparameters that we wish to estimate. 
Note that in Case II, $\bfy$ is often referred to as model parameters rather than hyperparameters, but for simplicity we denote all unknown parameters (except for the inverse parameters $\bfx$) as hyperparameters. 

We assume that the parameter sets $\bfx$ and $\bftheta$ are independent random variables, and furthermore, assume that $\bfpsi$ and $\bfy$ are also independent. 
Then using Bayes' theorem and combining \eqref{eq:case1} and \eqref{eq:case2}, the posterior density is given by
\begin{align}\label{eqn:posterior}
\pi(\bfx,\bftheta \mid \bfb)& = \frac{\pi (\bfb \mid \bfx, \bftheta)  \pi(\bfx\mid \bftheta) \pi_{\bftheta}(\bftheta)}{\pi(\bfb)} \notag\\
& \propto \frac{\pi_{\bfy} (\bfy) \pi_{\bfpsi}(\bfpsi) \exp \left( - \frac{1}{2} \| \bfb - \bfA(\bfy) \bfx \|^{2}_{\bfR(\bfpsi)^{-1}} - \frac{1}{2} \| \bfx - \bfmu_\bfx\|^{2}_{\bfQ_\bfx(\bfpsi)^{-1}} \right)}{\det(\bfR(\bfpsi))^{1/2} \det(\bfQ_\bfx(\bfpsi))^{1/2}},
\end{align}
where $\| \bfx \|_{\bfK} = \sqrt{\bfx^{\top} \bfK \bfx}$ defines a norm for any SPD matrix $\bfK.$
The marginal posterior density is obtained by integrating out $\bfx$, 
\begin{align}
\pi(\bftheta\mid\bfb) & = \int_{\bbR^n} \pi(\bfx,\bftheta \mid \bfb)d\bfx \notag \\
& \propto \pi_{\bfy}(\bfy) \pi_{\bfpsi}(\bfpsi)\det({\bfPsi}(\bftheta))^{-1/2} \exp \left( - \frac{1}{2} \| \bfA(\bfy) \bfmu_\bfx - \bfb \|^{2}_{{\bfPsi}(\bftheta)^{-1}} \right), 
\end{align}
where the hyperprior factorizes as $\pi_{\bftheta}(\bftheta)= \pi_{\bfy}(\bfy) \pi_{\bfpsi}(\bfpsi)$ and ${\bfPsi} : \bbR^{p} \rightarrow \bbR^{m \times m}$ is defined as
\begin{equation}
\label{eq:Zmatrixh}
{\bfPsi}(\bftheta) = \bfA(\bfy)\bfQ_\bfx(\bfpsi) \bfA(\bfy)\t + \bfR(\bfpsi).
\end{equation}

The problem of hyperparameter estimation seeks the MMAP estimate, and is obtained by minimizing the negative log of the marginal posterior, i.e.,
\begin{equation}
\label{eq:full_optimization}
\bftheta^{*} \in \arg\min_{\bftheta \in \Theta} \calF(\bftheta) \coloneqq -\ln{\pi(\bftheta\mid\bfb)},
\end{equation}
where $\Theta \subset \bbR^p$. 
{Expanding this objective function and dropping additive constants independent of $\bftheta$ yields:}
\begin{equation}
\label{eq:F}
    \mc{F}(\bftheta) = -\log\pi_{\bfy}(\bfy)-\log\pi_{\bfpsi}(\bfpsi) + \frac12 \logdet({\bfPsi}(\bftheta)) + \frac12 \|\bfA(\bfy) \bfmu_\bfx - \bfb\|_{{\bfPsi}(\bftheta)^{-1}}^2.
\end{equation}



For a fixed $\bfy$, the partial derivatives of $\mc{F}$ with respect to $\bfpsi$ ($1 \le j \le q$) are:
\begin{align}
\frac{\partial{\mc{F}}}{\partial\psi_j} = -&\frac{1}{\pi_{\bfpsi}(\bfpsi)}\frac{\partial\pi_{\bfpsi}(\bfpsi)}{\partial \psi_j} + \frac{1}{2}\trace\left({\bfPsi}(\bftheta)^{-1}\frac{\partial {\bfPsi}(\bftheta)}{\partial \psi_j}\right) \label{eq:fullGradpsi} \\
        &- \frac{1}{2}\Big[ {\bfPsi}(\bftheta)^{-1}\left(\bfA(\bfy)\bfmu_\bfx-\bfb\right)\Big]\t\left[\frac{\partial {\bfPsi}(\bftheta)}{\partial \psi_j}{\bfPsi}(\bftheta)^{-1}(\bfA(\bfy)\bfmu_\bfx-\bfb)  \right] \notag
\end{align}
where, for $1 \le j \le q$,
$$
\frac{\partial {\bfPsi}(\bftheta)}{\partial \psi_j} = \bfA(\bfy)\frac{\partial\bfQ_\bfx(\bfpsi)}{\partial\psi_j} \bfA(\bfy)\t + \frac{\partial\bfR(\bfpsi)}{\partial\psi_j}.
$$
Similarly, for a fixed ${\bfpsi}$, the partial derivatives with respect to $\bfy$ ($1 \le j \le \ell$) are:
\begin{align}
        & \frac{\partial{\mc{F}}}{\partial y_j} =  -\frac{1}{\pi_{\bfy}(\bfy)} \frac{\partial\pi_{\bfy}(\bfy)}{\partial y_j} + \frac{1}{2}\trace\left(\bfPsi(\bftheta)^{-1}\frac{\partial \bfPsi(\bftheta)}{\partial y_j}\right) \label{eq:fullGrady}   \\
         &- \frac{1}{2}\Big[ \bfPsi(\bftheta)^{-1}\left(\bfA(\bfy)\bfmu_\bfx-\bfb\right)\Big]\t\left[\frac{\partial \bfPsi(\bftheta)}{\partial y_j}\bfPsi(\bftheta)^{-1}(\bfA(\bfy)\bfmu_\bfx-\bfb)-2\frac{\partial\bfA(\bfy)}{\partial y_j}\bfmu_x  \right].  \notag
\end{align}
While the data misfit term $\frac12 \|\bfA(\bfy) \bfmu_\bfx - \bfb\|_{{\bfPsi}(\bftheta)^{-1}}^2$ can often be evaluated efficiently using iterative methods like preconditioned Conjugate Gradient (PCG), the log-determinant term $\log\det({\bfPsi}( {\bftheta}))$ requires $\mathcal{O}(m^3)$ operations plus the cost of forming $\bfPsi(\bftheta)$, which usually requires several PDE solves. This renders exact evaluation prohibitive for large-scale problems. One alternative approach is to use stochastic trace estimation techniques as in~\cite{chung2024efficient}, which we briefly recap.

\subsection{Stochastic Trace Estimation}
\label{sec:trace_est}
To bypass the computational bottleneck of the log-determinant term, we define the isolated term
\begin{equation}
\label{eq:tracelog}
    F(\bftheta) := \log\det({\bfPsi}(\bftheta)) =  \trace(\log({\bfPsi}(\bftheta))),
\end{equation}
and seek to approximate it efficiently and in a matrix-free manner. 

\paragraph{Hutchinson Estimator}
Given $N$ independent Rademacher vectors $\bfw_i \in \{\pm1\}^m$, the Hutchinson estimator for the log-determinant is defined as
\begin{equation}\label{e:hutch}
\widehat{F}_{N}(\bftheta) := \frac{1}{N} \sum_{i=1}^N \bfw_i^\top \log(\bfPsi(\bftheta))\, \bfw_i .
\end{equation}
This estimator is unbiased, with expectation $\mathbb{E}[\widehat{F}_{N}(\bftheta)] = F(\bftheta)$. 
For any symmetric $\bfB \in \mathbb{R}^{m\times m}$ and $\bfw \in \{\pm1\}^m$, we have 
$\bfw^\top \bfB \bfw - \trace(\bfB) = \bfw^\top \overline{\bfB} \bfw,$
where $\overline{\bfB} := \bfB - \diag(\bfB)$ denotes the off-diagonal part of $\bfB$. Consequently, the variance is $\mathbb{V}[\widehat{F}_{N}(\bftheta)] = \frac{2}{N}\|\overline{\log (\bfPsi(\bftheta))}\|_F^2$~\cite{avron2011randomized}. The diagonal of $\bfB$ is recovered exactly, and the error depends only on the off-diagonal entries.

\paragraph{Stochastic Lanczos Quadrature (SLQ)}
Evaluating $\widehat{F}_{N}(\bftheta)$ still requires computing the matrix logarithm. To obtain an efficient approximation, we compute approximations to the quadratic forms $\bfw_i^\top \log(\bfPsi(\bftheta)) \bfw_i$ using $K$ steps of the symmetric Lanczos algorithm \cite{Ubaru_SLQ}. For each $\bfw_i,$ where  $1 \le i \le N$, the method constructs an orthonormal basis for the Krylov subspace $\mathcal{K}_K(\bfPsi(\bftheta), \bfw_i)$. The Lanczos method implicitly yields a polynomial $q_K$ of degree at most $K-1$ such that 
$$\bfw_i^\top \log(\bfPsi(\bftheta))\,\bfw_i \approx \bfw_i^\top q_K(\bfPsi(\bftheta))\,\bfw_i.$$
The resulting SLQ estimator for $F(\bftheta)$ is defined as
\begin{align}
    \label{e:slq}
\widehat{F}_{\mathrm{SL}}(\bftheta) = \frac{1}{N}\sum_{i=1}^N \bfw_i^\top q_K(\bfPsi(\bftheta))\bfw_i,
\end{align}
Unlike the standard Hutchinson estimator, SLQ introduces a deterministic bias due to the polynomial approximation, which we analyze in Section \ref{sec:analysis}.

\subsection{SAA for estimating hyperparameters}\label{sub:SAA}
Stochastic trace estimation approaches such as those described in \Cref{sec:trace_est} have been used for hyperparameter estimation \cite{chung2024efficient}.  Using \eqref{eq:tracelog}, optimization problem \eqref{eq:full_optimization} can be expressed as a stochastic optimization problem,
\begin{equation}
\label{eq:expopt}
    \min_{\bftheta\in \Theta} -\log\pi_{\bftheta}(\bftheta) + \frac12 \expect_{\bfw}[\bfw\t\log(\bfPsi(\bftheta))\bfw] + \frac12 \|\bfA(\bfy) \bfmu_\bfx - \bfb\|_{\bfPsi(\bftheta)^{-1}}^2,
\end{equation}
and an SAA approach can be used, in conjunction with the SLQ estimator.
That is, we can draw $N$ independent realizations of Rademacher random vectors $\bfw$  and replace the expected value (from the log-determinant term) directly with the SLQ estimator $\widehat{F}_{\mathrm{SL}}(\bftheta)$. This yields the  objective:
\begin{equation} \label{eq:empirical_F}
    \widehat{\mc{F}}_{\mathrm{SL}}(\bftheta) := -\log\pi_{\bftheta}(\bftheta) + \frac12 \widehat{F}_{\mathrm{SL}}(\bftheta)+ \frac{1}{2} \|\bfA(\bfy) \bfmu_\bfx - \bfb\|_{{\bfPsi}(\bftheta)^{-1}}^2 .
\end{equation}

Then, one can compute the empirical minimizer $\displaystyle \widehat{\bftheta} \in \arg\min_{\bftheta \in \Theta} \widehat{\mc{F}}_{\mathrm{SL}}(\bftheta)$ as a proxy for the true minimizer $\bftheta^*$, see  \eqref{eq:full_optimization}. We now discuss a way to quantify the suboptimality of this solution. 
\subsection{Excess Risk}\label{ssec:excess} Let $\widehat{\mc{F}}$ be an approximation to $\mc{F}$ with respective minimizers $\widehat{\bftheta} \in \argmin_{\bftheta \in \Theta} \widehat{\mc{F}}(\bftheta)$ and $\bftheta^* \in \argmin_{\bftheta \in \Theta}\mc{F}(\bftheta)$. For this discussion, $\widehat{\mc{F}}$ is an arbitrary approximation that need not arise from the SAA approach.  The suboptimality of the optimizer is quantified by the excess risk $\mc{F}(\widehat{\bftheta}) - \mc{F}(\bftheta^*)$.

The excess risk of the empirical minimizer $\widehat{\bftheta}$ is bounded by the uniform approximation error of the trace estimator over $\Theta$~\cite[Remark 4.2]{saibaba2025stochastic}:
\begin{align}\label{eq:risk}
\begin{aligned}
0 \le \mc{F}(\widehat{\bftheta}) - \mc{F}(\bftheta^*) 
&\leq 2 \sup_{\bftheta \in \Theta} |\mc{F}(\bftheta) - \widehat{\mc{F}}(\bftheta)|.
\end{aligned}
\end{align}

This result shows that the reliability of $\widehat{\bftheta}$ hinges on the worst-case deviation of $\widehat{\mc{F}}$ from $\mc{F}$. In Section~\ref{sec:analysis}, we focus on analyzing $\sup_{\bftheta \in \Theta} |\mc{F}(\bftheta) - \widehat{\mc{F}}(\bftheta)|$ in the SAA and the \mmmc approaches.  

\section{Majorization-Minimization for Estimating Hyperparameters} \label{sec:MM-SAA}
In this section, we describe a new method for hyperparameter estimation called the \mmmc method that avoids the log determinant and matrix logarithm computation entirely.  Instead, we follow an MM principle, where a specific majorization function can be used to majorize the log determinant term, and we combine it with a Monte Carlo approximation of the trace term for efficient optimization.

\subsection{Background on MM Algorithms}The underlying idea of an MM algorithm is to replace a hard optimization problem with a sequence of simpler ones.  Consider minimizing the objective function $\mc{F}(\bftheta)$ using an iterative method where $\bftheta_t$ is the current iterate. Then, the MM principle majorizes the objective function $\mc{F}(\bftheta)$ by a surrogate function $\mc{G}(\bftheta \mid \bftheta_t)$.  Conditions for the majorization are the tangency condition
$\mc{G}(\bftheta_t \mid \bftheta_t) = \mc{F}(\bftheta_t)$
and the domination condition
$\mc{G}(\bftheta \mid \bftheta_t) \geq \mc{F}(\bftheta)$ for all $\bftheta \in \Theta \subset \bbR^p. $
Then the next iterate of the MM algorithm is given by a minimizer of the surrogate function, i.e., 
\begin{equation}
\label{eq:mmiterate}
    \bftheta_{t+1} \in \arg\min_{\bftheta \in \Theta } \mc{G}(\bftheta \mid \bftheta_t).
\end{equation}

The MM algorithm is guaranteed to be a descent algorithm since
\begin{equation}
    \label{eq:mmdescent}
\mc{F}(\bftheta_{t+1}) \leq \mc{G}(\bftheta_{t+1} \mid \bftheta_t) \leq \mc{G}(\bftheta_t \mid \bftheta_t) = \mc{F}(\bftheta_{t}).\end{equation}
In practice, solving \eqref{eq:mmiterate} does not need to be done exactly, since the descent property \eqref{eq:mmdescent} only relies on the subsequent iterate satisfying $\mc{G}(\bftheta_{t+1} \mid \bftheta_t) \le \mc{G}(\bftheta_t \mid \bftheta_t)$. With certain regularity conditions, an MM algorithm is guaranteed to converge to a stationary point of the objective function \cite{lange2016mm}.

For \eqref{eq:full_optimization}, we begin by defining a majorization function for $\mc{F}(\bftheta)$. Recall that evaluating the log determinant term can be very expensive, so we aim to replace this term with something simpler to evaluate.  With a view towards an inner-outer optimization scheme, let $\{\bftheta_t\}_{t \ge 0}$ denote the sequence of outer iterations and let $\bfPsi_t = \bfPsi(\bftheta_t)$.  Then, consider the following majorization function for the log determinant term \cite[Example 3.2.6 and Equation (4.20)]{lange2016mm},
\begin{equation}
\label{eq:logdetmajorization}
 \logdet(\bfPsi(\bftheta)) \leq \logdet(\bfPsi_t) + \trace [\bfPsi_t^{-1} (\bfPsi(\bftheta) - \bfPsi_t)] := \mc{Q}(\bftheta \mid \bftheta_t).
\end{equation}
Thus, a majorizer for $\mc{F}(\bftheta)$ is given by
\begin{equation}
\label{eq:Gmm}
    \mc{G}(\bftheta \mid \bftheta_t) :=  -\log\pi_{\bftheta}(\bftheta) + \frac12 \mc{Q}(\bftheta \mid \bftheta_t)+ \frac12 \|\bfA(\bfy) \bfmu_\bfx - \bfb\|_{\bfPsi(\bftheta)^{-1}}^2, 
\end{equation}
and it is easily verified that $\mc{G}(\bftheta_t \mid\bftheta_t) = \mc{F}(\bftheta_t)$. 
Given the current iterate $\bftheta_t$, the  MM approach generates the next iterate by solving the optimization problem $\displaystyle \min_{\bftheta \in \Theta} \mc{G}(\bftheta\mid \bftheta_t)$. Equivalently, we can solve 
\begin{equation}
\label{eq:mmiterate_full}
    \bftheta_{t+1}  \in \arg\min_{\bftheta \in \Theta} -\log\pi_{\bftheta}(\bftheta) + \frac12 \underbrace{\trace (\bfPsi_t^{-1} \bfPsi(\bftheta))}_{:=G(\bftheta \mid \bftheta_t)} + \frac12 \|\bfA(\bfy) \bfmu_\bfx - \bfb\|_{\bfPsi(\bftheta)^{-1}}^2
\end{equation}
where the terms independent of $\bftheta$ have been removed. 

\subsection{The \texorpdfstring{\mmmc}{M3C} Method} \label{sub:m3c} Similar to the approach described in Section~\ref{sub:SAA}, various randomized trace estimators can be used here, resulting in a Monte Carlo approximation of the majorizer at each outer MM iteration.  Specifically, at iteration $t$, let
$$ G(\bftheta \mid \bftheta_t) \approx \frac{1}{N_t} \sum_{i=1}^{N_t} \underbrace{\bfw_i\t \bfPsi_t^{-1}}_{=\bfz_{t,i}\t} \bfPsi(\bftheta) \bfw_i  = \widehat{G}_{N_t}(\bftheta \mid \bftheta_t),$$
where $\bfw_1,\cdots,\bfw_{N_t}$ are $N_t$ independent Rademacher vectors. Using the above estimation, we form the approximate majorant,
\begin{equation}\label{eqn:surrmc}
    \widehat{\mc{G}}_{N_t}(\bftheta\mid\bftheta_t) \equiv -\log\pi_{\bftheta}(\bftheta) + \frac{1}{2 N_t} \sum_{i=1}^{N_t} \bfz_{t,i}\t \bfPsi(\bftheta) \bfw_i + \frac12 \|\bfA(\bfy) \bfmu_\bfx - \bfb\|_{\bfPsi(\bftheta)^{-1}}^2.
\end{equation}
The gradient takes form (for $1 \le j \le p$)
\begin{align}
    \frac{\partial}{\partial \theta_j}\widehat{\mc{G}}_{N_t}(\bftheta&\mid\bftheta_t) 
    = -\frac{1}{\pi_{\bftheta}(\bftheta)}\frac{\partial\log\pi_{\bftheta}(\bftheta)}{\partial\theta_j} + \frac{1}{2 N_t} \sum_{i=1}^{N_t} \bfz_{t,i}\t \frac{\partial \bfPsi(\bftheta)}{\partial \theta_j} \bfw_i \label{eq:gradientMMMC}\\ -& \frac{1}{2}\Big[ \bfPsi(\bftheta)^{-1}\left(\bfA(\bfy)\bfmu_\bfx-\bfb\right)\Big]\t\left[\frac{\partial \bfPsi(\bftheta)}{\partial \theta_j}\bfPsi(\bftheta)^{-1}(\bfA(\bfy)\bfmu_\bfx-\bfb)-2\frac{\partial\bfA(\bfy)}{\partial \theta_j}\bfmu_x  \right]. \notag
\end{align}
A summary of the proposed approach for hyperparameter estimation denoted \mmmc can be found in Algorithm \ref{alg:MM}.

\begin{algorithm}[!ht]
\caption{\mmmc Algorithm for Hyperparameter Estimation}
\label{alg:MM}
\begin{algorithmic}[1]
    \Require Matrix $\mathbf{A}$, data vector $\mathbf{d}$, initial parameters $\bftheta_0 = \begin{bmatrix} \boldsymbol{\psi}_0^\top & \mathbf{y}_0^\top \end{bmatrix}^\top$

      
    \For{$t = 0, 1, 2, \ldots$} \Comment{Outer loop: Majorization}
        \State Solve: $\mathbf{\Psi}_t \mathbf{z}_{t,i} = \mathbf{w}_i $ for $i=1, \ldots, N_t$ 
        \Statex 

        \State {Update:} $\displaystyle \bftheta_{t+1} \gets
        \argmin_{\bftheta \in \Theta} \widehat{\mc{G}}_{N_t}(\bftheta\mid\bftheta_t) $ \Comment{Inner loop: Minimization}
    \EndFor
\end{algorithmic}
\end{algorithm}

The computational advantage of the \mmmc approach compared to the SAA approach is that we avoid the log determinant and hence the matrix logarithm entirely.  Although it may seem undesirable to introduce an inner-outer optimization scheme, the benefit is that the surrogate functions and their gradients are easier to evaluate.  Linear solves can be computed once and reused in multiple inner iterations.  Moreover, the surrogate optimization problem does not need to be solved exactly, so the number of inner iterations can be reduced.  Numerical investigations are provided in \Cref{sec:NumEx}. Additional constraints on the hyperparameters (e.g., bound constraints) can be included in the inner minimization problem.

\subsection{Computational Considerations}\label{ssec:compconsider}


Although there are a variety of ways to compute $\widehat{\mc{G}}_{N_t}(\bftheta\mid \bftheta_t)$ and its gradient, we consider one approach to outline the computational cost of the \mmmc method.
We assume that the cost of computing matvecs with $\bfA$ and its transpose is $T_{A}$ flops. Similarly, the cost of computing matvecs with $\bfQ$ and its derivatives is $T_Q$ flops.
We can assume that $\bfR$ is diagonal, so the cost of a matvec with it or its inverse is $\mc{O}(m)$ flops. With these assumptions, a matvec with $\bfPsi$ is $T_\Psi = 2T_A + T_Q + \mc{O}(m)$ flops. 
A discussion of the computational costs for the \mmmc approach follows:

\begin{enumerate}
    \item \textbf{Outer loop precomputation:} At each outer iteration $t$, $\bfz_{t,i} = \bfPsi_t^{-1}\bfw_i$ can be precomputed for all $1 \le i \le N_t, t\ge 0$ using PCG with an appropriately defined preconditioner $\bfG$ that satisfies $\bfG\t\bfG \approx \bfPsi(\bftheta)^{-1}$; see, e.g.,~\cite{chung2024efficient}, and reused during the inner iterations.  Let $T_G$ denote the cost of computing matvecs with $\bfG$, then the cost is $c_t N_tT_\Psi$ where $c_t$ is the total number of PCG iterations. It costs an additional $N_t((2c_t+2)T_G)$ flops to use a preconditioner: one matvec for $\bfG\bfw_i$ then 2 matvecs with $\bfG$ at every PCG iteration, followed by a matvec with $\bfG\t$.
    \item \textbf{Objective function:} Using the precomputed $\bfz_{t,i}$, the Monte Carlo trace estimation can be computed in $N_t(T_\Psi + \mc{O}(m))$ flops. The cost of computing $\bfr = \bfPsi(\bftheta)^{-1}(\bfA(\bfy)\bfmu_\bfx - \bfb)$ using PCG is $(c+2)T_G + cT_\Psi$ flops where $c$ is the number of PCG iterations. Then, an additional $\mc{O}(m)$ flops are needed to compute $(\bfA(\bfy)\bfmu_\bfx-\bfb)\t\bfr$.
    \item \textbf{Gradient computation:} 
    \begin{itemize}
        \item Option (a): Here we use the gradient given in \eqref{eq:gradientMMMC} where computations with $\frac{\partial \bfPsi(\bftheta)}{\partial\theta_j}$, $\frac{\partial \bfPsi(\bftheta)}{\partial\theta_j}$, and $\frac{\partial \bfA(\bfy)}{\partial\theta_j}$ are approximated using finite difference. Using the precomputed $\bfz_{t,i}$, we can approximate \, $\sum_{i=1}^{N_t}\bfz_{t,i}\t\frac{\partial \bfPsi(\bftheta)}{\partial\theta_j}\bfw_i$ in $p N_t(2T_\Psi + \mc{O}(m))$ flops. Reusing $\bfr$ from above, $\bfr\t\frac{\partial \bfPsi(\bftheta)}{\partial\theta_j}\bfr$ can be approximated in $p(2T_\Psi + \mc{O}(m))$ and $\bfr\t\frac{\partial \bfA(\bfy)}{\partial\theta_j}\bfmu_\bfx$ can be approximated in $p(2T_A + \mc{O}(m))$ flops.
        \item Option (b): Alternatively, we can use finite difference on the entire objective function. That is, $[\nabla \widehat{\mc{G}}_{N_t}(\bftheta\mid \bftheta_t)]_j \approx \frac{\widehat{\mc{G}}_{N_t}(\bftheta + \bfh_j\mid \bftheta_t) - \widehat{\mc{G}}_{N_t}(\bftheta\mid \bftheta_t)}{\epsilon}$ where $\bfh_j = \epsilon\bfe_j$ is sufficiently small. The cost of this option is $p(T_{\rm M^3C} + 2)$ where $T_{\rm M^3C}$ is the cost of one objective function evaluation.
    \end{itemize}
    
\end{enumerate}
Next we compare the computational costs of \mmmc with those of the SAA method, which can be found in the Supplementary Materials (\Cref{supp-sec:SAA}). The \mmmc outer loop precomputation has a similar cost to computing the log-determinant in the SAA approach, but is not required at every iteration. In total, the objective function evaluation for the \mmmc method is computationally cheaper than the SAA method. The cost of computing the gradient is the same for each method. However, for the \mmmc approach, the gradient requires more flops than the objective function evaluation, making it the main computational cost.

\section{Analysis}
\label{sec:analysis}
We provide convergence analysis for the SAA approach in \Cref{sub:SAAconvergence} and for the \mmmc approach in \Cref{sec:convergenceMMMC}.

\subsection{Analysis of the SAA Approach}
\label{sub:SAAconvergence}
We now establish the analysis of the SAA for the objective $\mc{F}(\bftheta)$, see \eqref{eq:F}. We begin by introducing regularity conditions on the parameter space $\Theta$ and the matrix-valued mapping $\bfPsi(\bftheta)$. These assumptions ensure that the mapping varies smoothly with respect to $\bftheta\in \Theta$. 

\begin{assumption}[Parameter Space]\label{ass:theta}
The parameter set $\Theta \subset \mathbb{R}^p$ is compact (i.e., closed and bounded), and therefore can always be enclosed by a closed Euclidean ball of radius $r > 0$. That is, there exists a center $\bftheta_c \in \mathbb{R}^p$ such that 
$$ \Theta \subseteq \left\{ \bftheta \in \mathbb{R}^p : \|\bftheta - \bftheta_c\|_2 \le r \right\}. $$
\end{assumption}

This boundedness assumption can be justified in two different ways. First, in some applications, the hyperprior has compact support (e.g., a uniform prior) and, thus, the parameter space is naturally bounded. Second, even if the hyperprior does not have compact support, some \textit{a priori} knowledge may be available for the hyperparameters, which can be safely enforced via a constrained optimization framework without loss of generality.

Next, we codify an assumption regarding the regularity of $\bfPsi$. Let $\mathbb{S}_{++}^m$ denote the cone of $m\times m$ SPD matrices. For two symmetric matrices $\bfC, \bfD$, the notation $\bfC \preceq \bfD$ (or $\bfD \succeq \bfC$) implies $\bfD-\bfC$ is positive semidefinite. Similarly, $\bfC \prec \bfD$ or $\bfD \succ \bfC$ implies $\bfD-\bfC$ is positive definite.

\begin{assumption}[Regularity of $\bfPsi$]\label{ass:psi}
The mapping $\bfPsi : \Theta \to \mathbb{S}_{++}^m$ satisfies:
\begin{enumerate}
\item[(i)] (Lipschitz continuity) 
There exists $L_\Psi > 0$ such that$$\|\bfPsi(\bftheta) - \bfPsi(\bftheta')\|_2 \le L_\Psi \|\bftheta - \bftheta'\|_2, \qquad \forall \, \bftheta, \bftheta' \in \Theta.$$
\item[(ii)] (Uniform spectral bounds) 
There exist constants 
$0 < \alpha \le \beta$ such that
$$\alpha I \preceq \bfPsi(\bftheta) \preceq \beta I, \qquad \forall \, \bftheta \in \Theta.$$
\end{enumerate}
We also denote $\kappa_\infty := \beta/\alpha$ as the uniform condition number.
\end{assumption}

Under Assumption~\ref{ass:psi}, the matrix logarithm $\log(\bfPsi(\bftheta))$ is well-defined for all $\bftheta \in \Theta$, and the identity \cref{eq:tracelog}
holds.
To establish uniform convergence over the space defined in Assumption~\ref{ass:theta}, we employ the covering number $N(\eta,\Theta,\|\cdot\|_2)$, defined as the minimal number of Euclidean balls of radius $\eta$ required to cover $\Theta$~\cite[Definition 4.2.2]{vershynin2018high}. For $\eta < r$, the covering number of a $p$-dimensional ball of radius $r$ satisfies\begin{align}\label{def:covTheta}N(\eta,\Theta,\|\cdot\|_2) \le \left(\frac{3r}{\eta}\right)^p .\end{align}

%

%



\paragraph{Lipschitz continuity of the matrix logarithm over $\Theta$}
Under the uniform spectral bound $\bfPsi(\bftheta) \succeq \alpha I$ for all 
$\bftheta \in \Theta$, see \Cref{ass:psi}, the matrix logarithm is Lipschitz continuous along the image of $\bfPsi$ as shown by the next result.

\begin{proposition}[Lipschitz continuity of matrix logarithm]
\label{prop:log-lipschitz}
Suppose $\bfPsi : \Theta \to \mathbb{S}_{++}^m$ satisfies
\(
\alpha I \preceq \bfPsi(\bftheta)
\)
for all $\bftheta \in \Theta$. Then, for any 
$\bftheta_1, \bftheta_2 \in \Theta$,
\[
\|\log (\bfPsi(\bftheta_1)) - \log (\bfPsi(\bftheta_2))\|_2
\le
\frac{1}{\alpha}
\|\bfPsi(\bftheta_1) - \bfPsi(\bftheta_2)\|_2.
\]
In particular, $\bftheta \mapsto \log (\bfPsi(\bftheta))$ is Lipschitz continuous with constant $L_\Psi/\alpha$.
\end{proposition}

\begin{proof}
The Fréchet derivative of the matrix logarithm has the integral representation \cite[Eq.~(11.10)]{HighamNicholas}
\[
L_{\mathrm{log}}(\bfB,\bfE) =
\int_0^1 (s(\bfB-\bfI)+\bfI)^{-1} \bfE (s(\bfB-\bfI)+\bfI)^{-1} \, ds.
\]
For any $\bfB \succeq \alpha \bfI \succ \mathbf{0}$, the matrix inside the inverse satisfies 
\[
s\bfB + (1-s)\bfI \succeq (s(\alpha-1)+1) \bfI\succ  \min\{1,\alpha\}\bfI.
\]
Since the induced 2-norm of an inverse SPD matrix is the reciprocal of its smallest eigenvalue, we have 
$\|(s(\bfB-\bfI)+\bfI)^{-1}\|_2 \le (s(\alpha-1)+1)^{-1}$.
Taking the norm of $L_{\mathrm{log}}(\bfB,\bfE)$ yields
\[
\|L_{\mathrm{log}}(\bfB,\bfE)\|_2 \le \|\bfE\|_2 \int_0^1 \frac{ds}{(s(\alpha-1)+1)^2} = \frac{1}{\alpha}\|\bfE\|_2.
\]

Now, let $\bfE = \bfPsi(\bftheta_1) - \bfPsi(\bftheta_2)$ and define the path 
$\bfB(t) = \bfPsi(\bftheta_2) + t\bfE$ for $t \in [0,1]$. 
Note that by convexity, $\bfB(t) \succeq \alpha \bfI$. 
We define the matrix-valued function $\bfG(t) = \log(\bfB(t))$. By the definition of the Fréchet derivative \cite[Sec.~3.1]{HighamNicholas}, the derivative of $\bfG$ with respect to  $t$ is
\[
\frac{d}{dt} \bfG(t) = L_{\mathrm{log}}(\bfB(t), \bfE).
\]
Using the fundamental theorem of calculus,  yields:
\[
\log (\bfPsi(\bftheta_1)) - \log (\bfPsi(\bftheta_2)) = \bfG(1) - \bfG(0) = \int_0^1 \frac{d}{dt} \bfG(t) \, dt = \int_0^1 L_{\mathrm{log}}( \bfB(t), \bfE ) \, dt.
\]
Taking the 2-norm of both sides, and applying our bound for $\|L_{\mathrm{log}}\|_2$ yields
\[
\|\log (\bfPsi(\bftheta_1)) - \log (\bfPsi(\bftheta_2))\|_2 \le \int_0^1 \| L_{\mathrm{log}}(\bfB(t), \bfE) \|_2 \, dt \le \int_0^1 \frac{1}{\alpha} \|\bfE\|_2 \, dt = \frac{1}{\alpha} \|\bfE\|_2,
\]
which completes the proof.
\end{proof}

We now state some results on the number of SLQ iterations and the concentration of the SLQ estimator, which are obtained from~\cite{CortinovisKressner}. We will use the notation that $\kappa(\bfM) = \|\bfM\|_2 \|\bfM^{-1}\|_2$, for an invertible matrix $\bfM$. 

\begin{lemma}[SLQ vs.\ log-det error]\label{lem:slq-vs-logdet}
Fix $\bftheta \in \Theta$ and define $\bfM := \bfPsi(\bftheta) \in \mathbb{S}_{++}^m$.
Let $\widehat{F}_N(\bftheta)$ and $\widehat{F}_{\mathrm{SL}}(\bftheta)$ be as in \eqref{e:hutch}--\eqref{e:slq}.
Then the following statements hold.

\begin{enumerate}
\item[(i)] \textbf{Deterministic SLQ--Hutchinson error.}
For any $\varepsilon > 0$, if
\begin{align}\label{def:k0}
k \;\ge\; \frac{\sqrt{\kappa(\bfM)+1}}{4}
\log\!\left(
4\varepsilon^{-1} m \bigl(\sqrt{\kappa(\bfM) +1}+ 1\bigr)
\log\bigl(2\kappa(\bfM)\bigr)
\right),
\end{align}
then
\[
\bigl| \widehat{F}_{N}(\bftheta) -\widehat{F}_{\mathrm{SL}}(\bftheta) \bigr|
\;\le\; \frac{\varepsilon}{2}.
\]
\item[(ii)] \textbf{Concentration of the SLQ estimator.}
Assume that $k$ satisfies \eqref{def:k0}, and that
\begin{align}
N \;\ge\; 32
\left(
 \,\varepsilon^{-2} \|\overline{\log \bfM} \|^2_F + \frac{\varepsilon^{-1}}{2}  \|\overline{\log \bfM} \|_2
\right)
\log\!\left(\frac{2}{\delta}\right).
\end{align}
Then, with probability at least $1-\delta$,
\[
|   \widehat{\mc{F}}_{\mathrm{SL}}(\bftheta) -\mc{F}(\bftheta)|
=
\bigl|
\widehat{F}_{\mathrm{SL}}(\bftheta)
-
\log\det\bigl(\bfPsi(\bftheta)\bigr)
\bigr|
\;\le\; \varepsilon.
\]

\end{enumerate}

\end{lemma}

\begin{proof}
Part (i) follows from \cite[Corollary~3]{CortinovisKressner}, and
 (ii)  from \cite[Theorem~5]{CortinovisKressner}.
\end{proof}

This result can be extended to the entire space $\Theta$ via a covering argument, following the approach in~\cite{saibaba2025stochastic}.

\paragraph{Uniform Error over $\Theta$}   
Since the stochasticity in the empirical objective $\widehat{\mc{F}}_{\mathrm{SL}}(\bftheta)$ in \eqref{eq:empirical_F} enters only through the log-determinant term, we focus on the convergence of its SAA approximation. Thus, we aim to bound the uniform approximation error 
$$\sup_{\bftheta \in \Theta} | \widehat{F}_{\mathrm{SL}}(\bftheta) - \log\det(\bfPsi(\bftheta)) |,$$
with high probability. For $0 < \eta < r$, the parameter space $\Theta$ is covered by an $\eta$-net $\mc{T} = \{\bftheta_1, \dots, \bftheta_S\}$ with cardinality $S \le (3r/\eta)^p$. We then select the number of Monte Carlo samples $N$ and Lanczos steps $K$ such that the SLQ estimator is accurate across all $\bftheta \in \mc{T}$ via a union bound. The result is then extended to the continuum $\Theta$ by leveraging the regularity of $\bfPsi$ from \Cref{ass:psi} to control the discretization error.

\begin{theorem}[Uniform SLQ Error]
\label{thm:uniform-slq}
Let $\delta \in (0,1)$ and $\varepsilon > 0$.
Suppose the SLQ estimator $\widehat{F}_{\mathrm{SL}}(\bftheta)$ is computed as in \eqref{e:slq} using $K$ Lanczos steps and $N$ independent Rademacher vectors. Under Assumptions~\ref{ass:theta} and~\ref{ass:psi}, if $K$ and $N$ satisfy
\begin{align}
    K \;\ge\;&
\frac{\sqrt{\kappa_\infty + 1}}{4}
\log\!\Bigl(
4\cdot \frac{5}{2}\, \varepsilon^{-1}
m (\sqrt{\kappa_\infty+1}+1)
\, \log\!\bigl(2\kappa_\infty\bigr)
\Bigr)\\
N \;\ge\; &
32
\left(\frac{25}{4}\varepsilon^{-2}  \max_{\bftheta\in \Theta}\|\overline{\log \bfPsi(\bftheta)} \|^2_F
+ \frac{5}{2}\varepsilon^{-1}\, \max_{\bftheta\in \Theta}\|\overline{\log \bfPsi(\bftheta)} \|_2
\right)\log\!\left(\frac{2\gamma}{\delta}\right),
\end{align}
where  $\gamma = \max\{(3r/\eta)^p,1\}$ and $\eta = \alpha \varepsilon / (5 m L_\Psi)$. 
Then, with probability at least $1-\delta$,
\[
\max_{\bftheta \in \Theta} 
\bigl| \widehat{F}_{\mathrm{SL}}(\bftheta) - \log\det(\bm{\Psi}(\bftheta)) \bigr|
\le \varepsilon.
\]

\end{theorem}
Before we present the proof, we discuss the implication of this theorem.
Let $ \widehat{\bftheta} \in \argmin_{\bftheta \in \Theta} \widehat{\mc{F}}_{\rm SL}(\bftheta)$ and let $\bftheta^* \in \argmin_{\bftheta \in \Theta} {\mc{F}}(\bftheta)$. By the arguments in Section~\ref{ssec:excess} and under the assumptions of Theorem~\ref{thm:uniform-slq}, we have with probability at least $1-\delta$, $$0 \le \mc{F}(\widehat{\bftheta}) - \mc{F}(\bftheta^*) \le 2 \max_{\bftheta \in \Theta} \frac12 |\widehat{F}_{\mathrm{SL}}(\bftheta) - \log\det(\bm{\Psi}(\bftheta)) | \le \varepsilon.$$ 
Thus, Theorem~\ref{thm:uniform-slq} provides the minimal number of SLQ iterations and the minimal number of Monte Carlo samples required for the absolute error in the excess risk to be bounded by $\varepsilon$ with high probability. 

\begin{proof} The proof is similar to~\cite[Theorem 4.3]{saibaba2025stochastic}, but additionally accounts for the error in the Lanczos approximation. 
Let $\mc{T} = \{\bftheta_1,\dots,\bftheta_S\}$ be an $\eta$-net of $\Theta$ in the Euclidean norm~\cite[Definition 4.2.1]{vershynin2018high} with
\(\eta = \alpha \varepsilon / (5 m L_\Psi)\) and $S \le \gamma$. 
Since  $\Theta$ is compact, there exists $\bftheta_h\in \Theta$ such that
\begin{align*}
    \max_{\bftheta\in \Theta}\left|\widehat{F}_{\mathrm{SL}}(\bftheta) - \log\det(\bm{\Psi}(\bftheta))\right|= \left|\widehat{F}_{\mathrm{SL}}(\bftheta_h) - \log\det(\bm{\Psi}(\bftheta_h))\right|.
\end{align*}
Let $\bftheta_0 \in \mc{T}$ be a net point such that $\|\bftheta_h - \bftheta_0\|_2 \le \eta$.
Applying the triangle inequality partitions the total error into three components:
\begin{align}\label{eqn:saainter1}
\left|\widehat{F}_{\mathrm{SL}}(\bftheta_h)- \log\det(\bfPsi(\bftheta_h))\right|
\leq & \underbrace{\left|\widehat{F}_{\mathrm{SL}}(\bftheta_h) - \widehat{F}_{\mathrm{SL}}(\bftheta_0)\right|}_{\equiv \gamma_1}+
\underbrace{\left|\widehat{F}_{\mathrm{SL}}(\bftheta_0) - \log\det(\bm{\Psi}(\bftheta_0))\right|}_{\equiv \gamma_2}\\ \nonumber
&+ \underbrace{\left|{\log\det}({\bfPsi}(\bftheta_h)) - \log\det({\bfPsi}(\bftheta_0))\right|}_{\equiv \gamma_3}.
\end{align}
We bound the three components $\gamma_3,\gamma_2$, and $\gamma_1$ in that order.
\paragraph{Control of the log-determinant variation}
By definition of the net, for any $\bftheta \in \Theta$ there exists $\bftheta_j \in \mc{T}$ such that $\|\bftheta - \bftheta_j\|_2 \le \eta$. 
Using  that $\trace(\bfA) \le m\|\bfA\|_2$ for $m\times m$ symmetric matrices, we bound the variation:
\begin{equation}
\label{eq:tracebound}
\gamma_3 = \|\trace(\log (\bfPsi(\bftheta_h)) - \log (\bfPsi(\bftheta_0)))\|_2
\leq m \| \log (\bfPsi(\bftheta_h)) - \log (\bfPsi(\bftheta_0))\|_2 \le \varepsilon/{5}.
\end{equation}
In the last step, we used the Lipschitz continuity of $\bm{\Psi}$ (Proposition~\ref{prop:log-lipschitz}) together with the spectral bound in \Cref{ass:psi}, to obtain
\begin{align}
\label{eq:lips}
\|\log (\bfPsi(\bftheta_h)) - \log (\bfPsi(\bftheta_0))\|_2 \le \frac{1}{\alpha}\| \bfPsi(\bftheta_h) -  \bfPsi(\bftheta_0)\|_2
\le \frac{ L_\Psi}{\alpha}\|\bftheta-\bftheta_j\|_2 \le \varepsilon/{5m}.
\end{align}
\paragraph{Estimator accuracy on the net}
Applying \Cref{lem:slq-vs-logdet} at each $\bftheta_j\in \mc{T},$ 
accuracy $\varepsilon/5$, and failure probability $\delta/S$, yields

\[
\mathbb{P}\bigl(|\widehat{F}_{\mathrm{SL}}(\bftheta_j)- \log\det({\bfPsi}(\bftheta_j))| > \varepsilon/5\bigr) \le \delta/{S},
\]
%
given the choices of $K$ and $N$ specified in the theorem statement. A union bound over all $S$ points in $\mathcal{T}$ then gives

\begin{align}
\label{eq:UnionBound}
\mathbb{P}\!\left(
\max_{1 \le j \le S}
\bigl|\widehat{F}_{\mathrm{SL}}(\bftheta_j)
- \log\det(\bm{\Psi}(\bftheta_j))\bigr|
\le \varepsilon/5
\right)
\ge 1-\delta.
\end{align}
Thus, with probability at least $1-\delta$, the second term  $\gamma_2$ is bounded by $\varepsilon/5$.
\paragraph{Estimator variation}
Finally, we expand the difference between the estimator evaluated at $\bftheta_h$ and $\bftheta_0$:
\begin{align*}
  &|F_{\mathrm{SL}}(\bftheta_h) - F_{\mathrm{SL}}(\bftheta_0)|\le E_1 +E_2+E_3,
 \end{align*}
 where 
 \begin{align*}
 E_1=& \frac{1}{N} \left| \sum_{i=1}^N \bfw_i^\top
 \left( 
 q_k(\bm{\Psi}(\bftheta_h)\! -\!\log\bm{\Psi}(\bftheta_h) 
 \right)\bfw_i\right|\\
 E_2=& 
 \frac{1}{N}\left| \sum_{i=1}^N
 \bfw_i^\top
 \left( 
\log\bm{\Psi}(\bftheta_0) \!-\! q_k(\bm{\Psi}(\bftheta_0) 
 \right)\bfw_i\right|\\
 E_3=& 
 \frac{1}{N}\left| \sum_{i=1}^N
 \bfw_i^\top
 \left( 
\log\bm{\Psi}(\bftheta_h)\!-\!
 \log\bm{\Psi}(\bftheta_0)
  \right)\bfw_i\right|.
 \end{align*}
The terms $E_1$ and $E_2$ correspond to the SLQ approximation error at 
$\bftheta_h$ and $\bftheta_0$, respectively, and are each bounded by 
$\varepsilon/5$ by \Cref{lem:slq-vs-logdet}, since $\kappa(\bftheta_h)\leq \kappa_\infty$. The remaining term $E_3$ can be controlled similarly to \eqref{eq:tracebound} since all Radamacher vectors have norm $\sqrt{m}$, yielding $E_3 \le \varepsilon/5$. Consequently, $
\gamma_1 = \bigl|\widehat{F}_{\mathrm{SL}}(\bftheta_h) - \widehat{F}_{\mathrm{SL}}(\bftheta_0)\bigr|
\le {3\varepsilon}/{5}$.
 
Plugging in the bounds for $\gamma_1,\gamma_2$, and $\gamma_3$ into~\eqref{eqn:saainter1} completes the proof.
\end{proof}




\subsection{Convergence Analysis of the \texorpdfstring{\mmmc}{M3C} Approach}
\label{sec:convergenceMMMC}
In this section, we analyze the convergence of the \mmmc scheme. 
We do this in three steps.  First, in Section \ref{subsub:general} we conduct a probabilistic analysis of MM approaches with inexactness in the optimal solution of each majorant.  This is derived for a general case, and hence may be of wider interest.  Then we focus on the \mmmc method, where we verify assumptions and continuity  properties in Section \ref{subsub:verify}.  The main result is provided in Section \ref{subsub:mmmc} where we show that as the number of iterations $t$ grows, the sequence of minimizers generated by the \mmmc iterations, with derived bounds on the minimal number of samples for the trace estimator, converges with high probability (at least $1-\delta$) to a first-order stationary point of the cost functional $\mc{F}(\bftheta)$.  

\subsubsection{General Case}
\label{subsub:general}
We first derive convergence analysis of MM algorithms for the objective function $\mathscr{F}$ with the assumption that the surrogate function $\mathscr{G}$ is approximated by a surrogate $\widehat{\mathscr{G}}_t$ at iteration $t$.  
Throughout the following analysis, the notation $\nabla_1 \mathscr{G}$ denotes the partial gradient of the surrogate function with respect to its first argument, treating the anchor point as a fixed parameter.

We clarify the assumptions needed on $\mathscr{{F}}$ and $\mathscr{G}$. 
\begin{assumption}\label{ass:funcgrad}
 
Assume the objective function  $\mathscr{F}: \mc{U}\subset \mathbb{R}^d \to \mathbb{R}$ is continuously differentiable, and  the surrogate function $\mathscr{G}: \Theta \times \Theta \to \mathbb{R}$ satisfies the following for all $\bftheta \in \Theta$:

\begin{enumerate}
\item \textbf{Majorization:} $\mathscr{G}(\bftheta \mid \bftheta') \ge \mathscr{F}(\bftheta)$ for fixed $ \bftheta' \in \Theta$
\item \textbf{Strong Tangency:} $\mathscr{G}(\bftheta \mid \bftheta) = \mathscr{F}(\bftheta)$ and $\nabla_1 \mathscr{G}(\bftheta \mid \bftheta) = \nabla \mathscr{F}(\bftheta)$.
\item \textbf{Regularity:} The gradient $\nabla_1 \mathscr{G}(\cdot \mid \bftheta')$ is $L$-Lipschitz continuous on $\Theta$ for any fixed $\bftheta' \in \Theta$.
\end{enumerate}
   
\end{assumption}

We now present the convergence analysis for the \mmmc method. Our proof builds on elements in~\cite[Proposition 2.8]{MMAlgorithmsLange}, but is different in two important ways:  (1) it incorporates inexactness in the surrogate and (2) it applies to constrained optimization problems.

\begin{theorem}
\label{thm:convMMSAA}
Let $\Theta$ satisfy Assumption~\ref{ass:theta} and additionally be convex. Let  $\mc{U}$ be an open neighborhood of $\Theta$ and instate Assumption~\ref{ass:funcgrad}. 

Let  $\{\widehat{\bftheta}_t\}_{t \ge 0}$ be the sequence of the \mmmc iterates, generated by: 
\[
\widehat{\bftheta}_{t+1} \in \arg\min_{\bftheta \in \Theta} \widehat{\mathscr{G}}_t(\bftheta \mid \widehat{\bftheta}_t).
\]
where $\widehat{\mathscr{G}}_t$ is some stochastic approximation of the surrogate $\mathscr{G}$. 

Let $\{\varepsilon_t\}_{t \ge 0}$ and $\{\delta_t\}_{t \ge 0}$ be sequences of strictly positive error tolerances and local failure probabilities satisfying $\displaystyle \sum_{t=0}^{\infty} \varepsilon_t < \infty$ and $\displaystyle \sum_{t=0}^{\infty} \delta_t \le \delta$ for some target global failure probability $\delta \in (0,1)$. 
If, at each iteration $t$, the  surrogate uniformly approximates the MM surrogate such that
\begin{equation}\label{eq:coverGN}
\mathbb{P} 
\left( \max_{\bftheta \in \Theta} \left| \mathscr{G}(\bftheta \mid \widehat{\bftheta}_t) - \widehat{\mathscr{G}}_t(\bftheta \mid \widehat{\bftheta}_t) \right| \le 
\frac{\varepsilon_t}{2} \left| \widehat{\bftheta}_t \right. \right) \ge 1 - \delta_t,
\end{equation} 
%
Then, with probability at least $1 - \delta$, every accumulation point $\bftheta^*$ of the empirical sequence $\{\widehat{\bftheta}_t\}_{t \ge 0}$ is a first-order stationary point of $\mathscr{F}$ over $\Theta$; that is,
\[
    \langle \nabla \mathscr{F}(\bftheta^*), \bftheta - \bftheta^* \rangle \ge 0 \quad \forall \bftheta \in \Theta.
\]
\end{theorem}
\begin{proof}
Let $\{\bftheta_t\}_{t \ge 0}$  be the sequences of the  MM  iterates,  generated by: 
\[
\bftheta_{t+1} \in \arg\min_{\bftheta \in \Theta} \mathscr{G}(\bftheta \mid \widehat{\bftheta}_t).
\]

By the excess risk bound \eqref{eq:risk} and the uniform bound  assumption \eqref{eq:coverGN}, for each $t \ge 0$ the local event
\begin{equation}
    \mc{E}_t = \left\{ \mathscr{G}(\widehat{\bftheta}_{t+1} \mid \widehat{\bftheta}_t) - \mathscr{G}(\bftheta_{t+1} \mid \widehat{\bftheta}_t) \le \varepsilon_t \left|  \widehat{\bftheta}_t \right.\right\},
\end{equation}
holds with probability at least $1 - \delta_t$.

We define the global success event as $\mc{E}_{\mathrm{total}} = \bigcap_{t=0}^{\infty} \mc{E}_t$. By the union bound, the probability of failure is bounded by:$$\mathbb{P}(\mc{E}_{\mathrm{total}}^c) \le \sum_{t=0}^{\infty} \mathbb{P}(\mc{E}_t^c) \le \sum_{t=0}^{\infty} \delta_t \le \delta.$$
Thus, $\mc{E}_{\mathrm{total}}$ occurs with probability at least $1 - \delta$. We condition the remainder of the proof on this event.

The majorization property \cite[Prop.~2.8]{MMAlgorithmsLange} implies that the progress in the target objective satisfies
\begin{align}
    \mathscr{F}(\widehat{\bftheta}_{t+1}) - \mathscr{F}(\widehat{\bftheta}_t) 
    &\le \mathscr{G}(\widehat{\bftheta}_{t+1} \mid \widehat{\bftheta}_t) - \mathscr{G}(\widehat{\bftheta}_t \mid \widehat{\bftheta}_t) \notag \\
    &\le \mathscr{G}(\bftheta_{t+1} \mid \widehat{\bftheta}_t) - \mathscr{G}(\widehat{\bftheta}_t \mid \widehat{\bftheta}_t) + \varepsilon_t.
\end{align}

Using a quadratic upper bound (see, e.g.,~\cite[Section 2]{MMAlgorithmsLange}) for $\mathscr{G}(\cdot\mid \widehat{\bftheta}_t)$, which is L-smooth with Lipschitz constant $L$, for any $\bftheta\in \Theta$,
\begin{equation}
    \mathscr{F}(\widehat{\bftheta}_{t+1}) - \mathscr{F}(\widehat{\bftheta}_t) 
    \le \langle \nabla \mathscr{G}(\widehat{\bftheta}_t \mid \widehat{\bftheta}_t), \bftheta - \widehat{\bftheta}_t \rangle + \frac{L}{2} \|\bftheta - \widehat{\bftheta}_t\|_2^2 + \varepsilon_t.
\end{equation}

From the strong tangency condition, this simplifies to:
\begin{equation}
\label{eq:smoothnessInequality}
    \mathscr{F}(\widehat{\bftheta}_{t+1}) - \mathscr{F}(\widehat{\bftheta}_t) 
    \le \langle \nabla \mathscr{F}(\widehat{\bftheta}_t), \bftheta - \widehat{\bftheta}_t \rangle + \frac{L}{2} \|\bftheta - \widehat{\bftheta}_t\|_2^2 + \varepsilon_t.
\end{equation}
Without any loss of generality, assume $\nabla \mathscr{F}(\widehat{\bftheta}_t) \neq \bfzero$ for any $t\ge 0$; otherwise, stationarity is trivially satisfied. Define the projected gradient step

\[
\bfz_t = \widehat{\bftheta}_t - L^{-1} \nabla \mathscr{F}(\widehat{\bftheta}_t).
\]
As $\Theta$ is closed and convex, the projected iterate is uniquely defined by $\widecheck{\bftheta}_{t+1} = \mc{P}_{\Theta}(\bfz_t)$, where $\mc{P}_{\Theta}$ denotes the Euclidean projection onto $\Theta$.
The optimality condition for the projection, see \cite[Sec. 8.1.3]{BoydConvex}, implies that for all $\bftheta \in \Theta$:
\begin{align}
\label{def:convexopt}
   \langle \bfz_t-\widecheck{\bftheta}_{t+1},\bftheta-\widecheck{\bftheta}_{t+1} \rangle=  \langle ( \widehat{\bftheta}_t - L^{-1} \nabla \mathscr{F}(\widehat{\bftheta}_t))-\widecheck{\bftheta}_{t+1},
   \bftheta-\widecheck{\bftheta}_{t+1}\rangle\leq 0,
\end{align}
Substituting $\bftheta=\widehat{\bftheta}_{t}$, we obtain 
\begin{align}
     \left \langle  \nabla \mathscr{F}(\widehat{\bftheta}_t),\widecheck{\bftheta}_{t+1} -\widehat{\bftheta}_{t} \right\rangle
     \leq -L \left\|\widehat{\bftheta}_{t}-\widecheck{\bftheta}_{t+1} \right\|^2.  
\end{align}
Combining  this bound with the inequality \eqref{eq:smoothnessInequality} 
for $\bftheta=\widecheck{\bftheta}_{t+1}$ 
 then yields
\begin{equation}
    \mathscr{F}(\widehat{\bftheta}_{t+1}) - \mathscr{F}(\widehat{\bftheta}_t) \le
    - \frac{L}{2}  \left\|\widehat{\bftheta}_{t}-\widecheck{\bftheta}_{t+1} \right\|^2 + \varepsilon_t.
\end{equation}
Summing this inequality from $t=0$ to $T$ gives:
\begin{equation}
\frac{L}{2}\sum_{t=0}^{T}
\|\widecheck{\bftheta}_{t+1} - \widehat{\bftheta}_t\|_2^2
\le
\mathscr{F}(\widehat{\bftheta}_0) - \mathscr{F}(\widehat{\bftheta}_{T+1})
+ \sum_{t=0}^{T} \varepsilon_t.
\end{equation}
Since $\mathscr{F}$ is continuous and $\Theta$ is compact, $\mathscr{F}$ is bounded below, and  $\displaystyle \sum_{t\ge 0}\varepsilon_t<\infty,$ letting $T \to \infty$ implies that the sum on the left converges, forcing:
$\displaystyle \lim_{t \to \infty} \|\widecheck{\bftheta}_{t+1} - \widehat{\bftheta}_t\|_2 = 0.$
By the compactness of $\Theta$, the sequence $\{\widehat{\bftheta}_t\}_{t\ge 0}$ admits at least one accumulation point. Let $\bftheta^*$ be any such point, and let $\{\widehat{\bftheta}_{t_k}\}_{k \ge 0}$ be a subsequence converging to $\bftheta^*$.
By the continuity of 
  $\nabla \mathscr{F}$,  taking the limit as $k \to \infty$ in \eqref{def:convexopt}  yields 
\[
\langle
\nabla \mathscr{F}(\bftheta^\star),
\bftheta - \bftheta^\star
\rangle \ge 0,
\quad \forall \bftheta\in\Theta.
\]
\end{proof}
 In \Cref{thm:convMMSAA}, we can choose the 
sequences of error tolerances $\{\varepsilon_t\}_{t\ge 0}$ and local failure probabilities $\{\delta_t\}_{t\ge 0}$ can be chosen to decay geometrically as
\[
    \varepsilon_t = \varepsilon_t \rho^t \quad \text{and} \quad \delta_t = \delta_0 \rho^t, \quad t \ge 0
\]
where $\varepsilon_0 > 0$, $\rho \in (0, 1)$, and $\delta_0 \in (0, \delta(1-\rho)]$.
It is readily verified that the summability conditions 
$\displaystyle \sum_{t = 0}^\infty \varepsilon_t < \infty$ and $\displaystyle \sum_{t = 0}^\infty \delta_t \le \delta$ are naturally satisfied.

\subsubsection{Regularity of MM surrogate}
\label{subsub:verify}
 
%

We study the surrogate function $\mc{G}(\bftheta \mid \bftheta_t)$ for $\mc{F}(\bftheta)$ used in the MM updates, see \eqref{eq:Gmm}. By construction, this function is a majorant that satisfies the strong tangency condition. To apply Theorem~\ref{thm:convMMSAA}, we need to establish the regularity properties of this surrogate function. To this end, we first introduce the following assumptions regarding the objective's components.

\begin{assumption}[$L$-smoothness] \label{assump:regularity}
The following functions possess Lipschitz continuous gradients on the parameter space $\Theta$: the mapping $  \bfA(\bftheta)\bfmu_\bfx$, the matrix mapping $\bfPsi(\bftheta)$, and the log-prior density $\log \pi_{\bftheta}(\bftheta)$.
\end{assumption}
Under these conditions, we can guarantee the smoothness of the surrogate gradient. In particular,  for any fixed $\bftheta_t \in \Theta$, the gradient of the surrogate function with respect to its first argument, $\nabla_1\mc{G}(\bftheta \mid \bftheta_t)$, is Lipschitz continuous on $\Theta$.

\begin{lemma}[$L$-smoothness of the MM surrogate]
\label{cor:LsmoothnessG}
Let $\mc{F}(\bftheta)$ be the functional defined in \eqref{eq:F}, and let 
$\mc{G}(\bftheta \mid \bftheta_t)$ be the MM surrogate introduced in \eqref{eq:Gmm}. 
Then, under Assumption~\ref{assump:regularity} the gradient of the surrogate function $\nabla_1 \mc{G}(\cdot \mid \bftheta_t)$ is $L$-Lipschitz continuous on $\Theta$. 

\end{lemma}
The proof of this result uses standard arguments and is deferred to the Supplementary materials (Section~\ref{supp-sec:proofs}).

\subsubsection{Convergence Analysis for the \texorpdfstring{\mmmc}{M3C} method}
\label{subsub:mmmc}
Consider the problem of minimizing a target function 
\(
\mc{F}:\Theta \to \mathbb{R}
\)
over the convex and  compact set $\Theta \subset \mathbb{R}^p$. In the MM framework, the optimization is performed by iteratively minimizing a surrogate function $\mc{G}(\bftheta \mid \bftheta_t)$ that majorizes the objective.
In the \mmmc approach, the exact surrogate $\mc{G}(\bftheta \mid \bftheta_t)$ is not evaluated  or optimized directly. Instead, we use a Monte Carlo approximation denoted by $\widehat{\mc{G}}_{N_t}(\bftheta \mid \bftheta_t)$, constructed from $N_t$ independent samples at iteration $t \ge 0$.
%
Let
\[
\bftheta_{t+1} \in \arg\min_{\bftheta \in \Theta} \mc{G}(\bftheta \mid \widehat{\bftheta}_t), 
\qquad
\widehat{\bftheta}_{t+1} \in \arg\min_{\bftheta \in \Theta} \widehat{\mc{G}}_{N_t}(\bftheta \mid \widehat{\bftheta}_t).
\]
Although the approximate surrogate $\mc{G}_{N_t}(\bftheta \mid \widehat{\bftheta}_t)$ depends parametrically on the current iterate $\widehat{\bftheta}_t$, the sample size $N_t$ can be chosen uniformly and independently of $\widehat{\bftheta}_t$ so that the resulting estimator satisfies a uniform excess-risk bound. We now present a result for the convergence analysis of the M$^3$C method applied to hyperparameter estimation.

\begin{theorem}\label{thm:coverMMSAA}
Let $\{\varepsilon_t\}_{t\ge 0}$ and $\{\delta_t\}_{t \ge 0}$ be two positive sequences satisfying \, $\sum_{t=0}^\infty \varepsilon_t < \infty$ and $\sum_{t=0}^\infty \delta_t = \delta < 1$  and define the sequence for $t \ge 0$
\[
\gamma_t = \max\left\{\left(\frac{12rmL_\Psi}{\varepsilon_t\alpha}\right)^p,1\right\},
\]
where  $\Theta$ satisfies Assumption~\ref{ass:theta} and is convex. Furthermore, suppose Assumptions~\ref{ass:psi} and~\ref{assump:regularity} hold. If the number of samples at iteration $t$, denoted $N_t$, satisfies
\[
N_t \ge 
\frac{16(2\varsigma_F^2+\varepsilon_t\varsigma_2)}{\varepsilon_t^2}
\ln\!\left(\frac{2\gamma_t}{\delta_t}\right),
\]
where $\displaystyle \varsigma_F:=\frac{1}{\alpha}\max_{\bftheta\in \Theta} \| {\bfPsi(\bftheta)}\|_F$, and
$\displaystyle \varsigma_2:=\frac{1}{\alpha}\max_{\bftheta\in \Theta} \|{\bfPsi(\bftheta)}\|_2$.
     Then, with probability at least $1-\delta$, any accumulation point of the sequence $\{\widehat{\bftheta}_t\}_{t\ge 0}$ is a first order stationary point of $\mc{F}$, with probability at least $1-\delta$. 

\end{theorem}
\begin{proof}
The plan is to apply Theorem 4.3 in \cite{saibaba2025stochastic} to establish that the events
\begin{equation}\label{eqn:event_et}
    \mc{E}_t = \left\{ 
   \mc{G}( \widehat{\bftheta}_{t+1}\mid \widehat{\bftheta}_t)
-
\mc{G}( {\bftheta}_{t+1}\mid \widehat{\bftheta}_t) \le \varepsilon_t \left| \widehat\bftheta_t \right. \right\},
\end{equation}
hold with probability at least $1-\delta_t$ for $t\ge 0$, and then apply  Theorem~\ref{thm:convMMSAA}. 
 Consider the decomposition
$
\mc{G}(\bftheta\mid\widehat{\bftheta}_t)
=
\trace(\bfK(\bftheta))+R(\bftheta),
$ where
\begin{align*}
\bfK(\bftheta)
&=
\frac{1}{2}\!\left(
\bfPsi^{-1}(\widehat{\bftheta}_t)\bfPsi(\bftheta)+ \bfPsi(\bftheta)\bfPsi^{-1}(\widehat{\bftheta}_t)
\right),\\
R(\bftheta)
&=
-\log\pi_{\bftheta}(\bftheta)
+
\frac{1}{2}\left(
\log\det(\bfPsi(\widehat{\bftheta}_t))
-
\trace(\bfI)
\right)
+
\frac{1}{2}
\big\|
\bfA(\bftheta)\bm{\mu}_{\bfx}-\bfb
\big\|^2_{\bfPsi^{-1}(\bftheta)} .
\end{align*}
Under \Cref{ass:psi}, we have 
\(
\alpha \bfI \preceq \bfPsi(\bftheta),
\)
which implies  $\|\bfPsi^{-1}(\widehat{\bftheta}_t)\|_2 \le 1/\alpha$. 
Using the submultiplicativity of the spectral and Frobenius norms, we bound $\bfK(\bftheta)$ as:
\begin{align*}
\|\bfK(\bftheta)\|_\xi \le \|\bfPsi^{-1}(\widehat{\bftheta}_t)\|_2 \|\bfPsi(\bftheta)\|_\xi \le \frac{1}{\alpha} \|\bfPsi(\bftheta)\|_\xi, \quad \xi \in \{2, F\}.
\end{align*}
The concentration analysis in \cite{saibaba2025stochastic} requires a bound on the off-diagonal part,
$\overline{\bfK}(\bftheta) = \bfK(\bftheta) - \diag(\bfK(\bftheta))$. By~\cite{bhatia1989comparing}, 
\begin{align*}
\|\overline{\bfK}(\bftheta)\|_\xi \le   \|\bfK(\bftheta)\|_\xi \le \frac{1}{\alpha} \|\bfPsi(\bftheta)\|_\xi,  \quad \xi \in \{2, F\},
\end{align*}
which yields the constants $\varsigma_F$ and $\varsigma_2$ used in the sample complexity bound.

By the choice of $N_t$ and the application of the uniform concentration bound for sample average approximations, see \cite[Theorem 4.3]{saibaba2025stochastic} 
and Section~\ref{supp-sec:StochasticSurrogate}, the events $\mathcal{E}_t$ defined in~\eqref{eqn:event_et} hold with probability at least $1-\delta_t$ for $t \ge 0$. We conclude by applying the convergence result in \Cref{thm:convMMSAA}.  \Cref{assump:regularity} and \Cref{cor:LsmoothnessG} guarantee the necessary $L$-smoothness of the surrogate. Since $\Theta$ is compact and convex, and the approximation errors are summable, the conditions of \Cref{thm:convMMSAA} are satisfied. Thus, by Theorem~\ref{thm:convMMSAA}, every accumulation point of the sequence $\{\widehat{\bftheta}_t\}_{t\ge 0}$ is a first-order stationary point of $\mc{F}$ with the prescribed probability.
\end{proof}

Theorem~\ref{thm:coverMMSAA} requires $\varepsilon_t \to 0$ as $t \rightarrow \infty$, forcing the sample size to grow with the iteration index $t$. Although the samples can be reused, each iteration solves a different SAA problem.

\section{Numerical Results} \label{sec:NumEx}
In Section~\ref{sub:seismic}, we consider an example of Case I, where the hyperparameters define the noise variance, the prior variance, and the correlation length of the prior. In Sections~\ref{sec:sr} and~\ref{sec:ns}, we consider two examples of Case II: super-resolution imaging and contaminant source identification, where both problems can be represented as separable nonlinear inverse problems \eqref{eq:case2}, and model hyperparameters can be estimated using the proposed \mmmc approach.

Analogous to the two approaches for gradient approximation for \mmmc (see \Cref{ssec:compconsider}), we consider two approaches for gradient approximation for SAA. For Option (a), the gradients of $\mc{F}(\bftheta)$, found in \eqref{eq:fullGradpsi} and \eqref{eq:fullGrady}, are approximated using a Hutchinson trace estimator of $\trace( \bfPsi(\bftheta)^{-1}\frac{\partial \bfPsi(\bftheta)}{\partial\theta_j})$, combined with finite difference on the terms $\frac{\partial \bfPsi(\bftheta)}{\partial\theta_j}$, $\frac{\partial \bfPsi(\bftheta)}{\partial\theta_j}$, and $\frac{\partial \bfA(\bfy)}{\partial\theta_j}$. Additional details on this approach are provided in the Supplementary materials (see Section~\ref{supp-sec:SAA}). For Option (b), finite difference is used on the entire objective function $\widehat{\mc{F}}_{\rm SL}(\bftheta)$. That is, $[\nabla \widehat{\mc{F}}_{\rm SL}(\bftheta)]_j \approx \frac{\widehat{\mc{F}}_{\rm SL}(\bftheta + \bfh_j) - \widehat{\mc{F}}_{\rm SL}(\bftheta)}{\epsilon}$ where $\bfh_j = \epsilon\bfe_j$ is sufficiently small.

\subsection{Seismic Inversion}
\label{sub:seismic}
We begin with a model problem from seismic inversion, where the goal is to reconstruct an image depicting the slowness of the subsurface using seismic waves. 
We construct an instance of the problem using the \verb|PRseismic| function in IR Tools~\cite{gazzola2019ir} resulting in a linear inverse problem of the form \cref{eq:case1}
where $\bfA\in\bbR^{1,440\times 65,536}$ is the forward model, $\bfx\in\bbR^{65,536}$ is the vectorized $256\times 256$ unknown image provided in \Cref{fig:SIrecons}, $\bfb\in\bbR^{1,440}$ are observations, and $\bfe\in\bbR^{1,440}$ represents the noise.
The number of measurements $m$ is obtained from $m = 32 \times  45 = 1440$, where $32$ is the number of source rays and $45$ is the number of receivers.  The observations are generated using some $\bfx_{\rm true}$ and to simulate measurement noise, we add $2\%$ Gaussian white noise. We define the prior and likelihood distributions using $\bfR(\bftheta)=\theta_1\bfI$, where $\theta_1$ is unknown, $\bfmu_\bfx = \bfzero$, and
$\bfQ_\bfx(\bftheta)$ is defined by a Mat\'ern kernel, 
\begin{equation}
\label{eq:Matern}
{\rm matern}(\bfr,\bfr') = \theta_2^2 \frac{2^{1-\nu}}{\Gamma(\nu)}\left( \sqrt{2\nu}\frac{\|\bfr-\bfr'\|_2}{\theta_3} \right)^\nu K_\nu\left( \sqrt{2\nu}\frac{\|\bfr-\bfr'\|_2}{\theta_3}\right),
\end{equation}
where $K_\nu$ is the modified Bessel function of the second kind, $\theta_2^2$ controls the variance of the process, $\theta_3$ is the length scale, and $\nu=\frac{1}{2}$ is the smoothness. The prior distribution for each component of the unknown hyperparameters\footnote{Note that for Case I, $\bftheta = \bfpsi$.}, $\bftheta=(\theta_1,\theta_2,\theta_3)$,  is the gamma distribution 
$\theta \sim \Gamma(\beta)\propto  e^{-\beta\theta}$, where $\beta=1\times10^{-4}$. Note that although the true value of $\theta_1^2=1.0762\times10^{-4}$ is known, the true values of $\theta_2$ and $\theta_3$ are not.


We set the initial value to be $\bftheta_0 =(3, 0.25, 1)$. The reconstruction at $\bftheta_0$ is provided in \Cref{fig:SIrecons}. For both SAA and \mmmc methods, we use the fmincon interior point method with lower bound $l_b = (1\times10^{-7},\ 1\times10^{-7},\ 1\times10^{-7})$, given that parameters $\theta_2$ and $\theta_3$ must be positive, and upper bound $u_b = (100,\ 100,\ 100)$. For the SAA approach, the maximum number of iterations is 200 and the step tolerance is $1e-4$. For the \mmmc approach, the maximum number of inner iterations is set to 2, 5, 10, and 15, and the step tolerance is set to $1e-4$. For the outer-loop, the maximum number of iterations is 200 and the stopping criteria is $\|\widehat\bftheta_t-\widehat\bftheta_{t+1} \|_2/\| \widehat\bftheta_{t+1}\|_2 < 1e-4$. For each setup, we use $N_t=24$ independent Rademacher vectors for the Monte Carlo trace estimation.
Also, for the \mmmc method, the preconditioner described in \cite{chung2024efficient} is formed at every outer iteration and reused for each inner iteration.

\begin{figure}[ht!]
\centering
\begin{tabular}{c c c c c}
    Truth & $\bftheta_0$ & SAA & \mmmc (2) & \mmmc (5) \\
    \includegraphics[width=.16\textwidth]{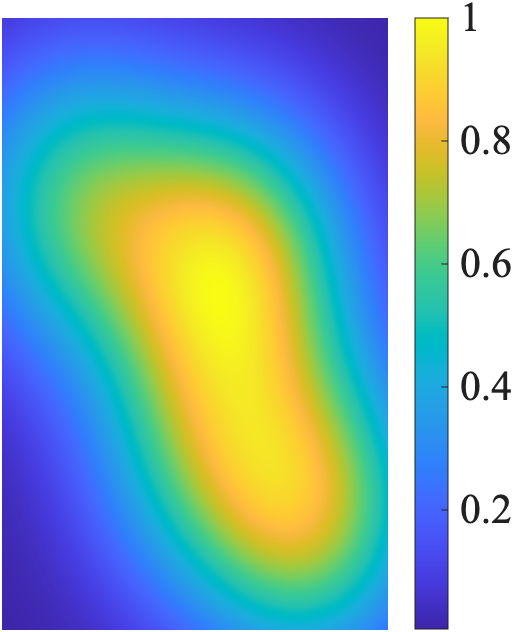} 
    & \includegraphics[width=.16\textwidth]{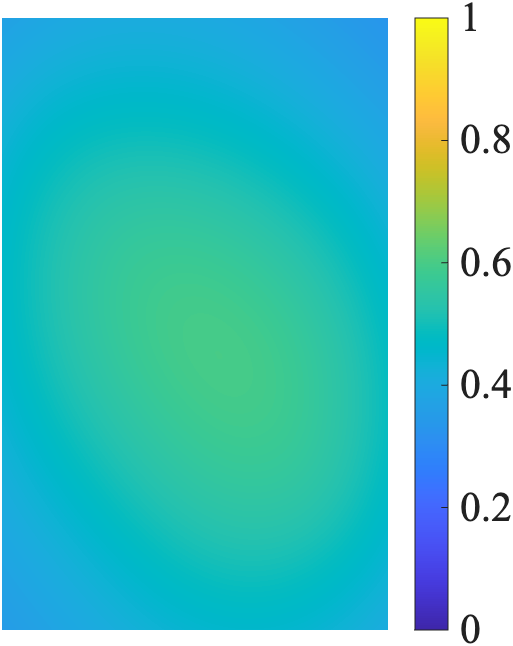}
    & \includegraphics[width=.16\textwidth]{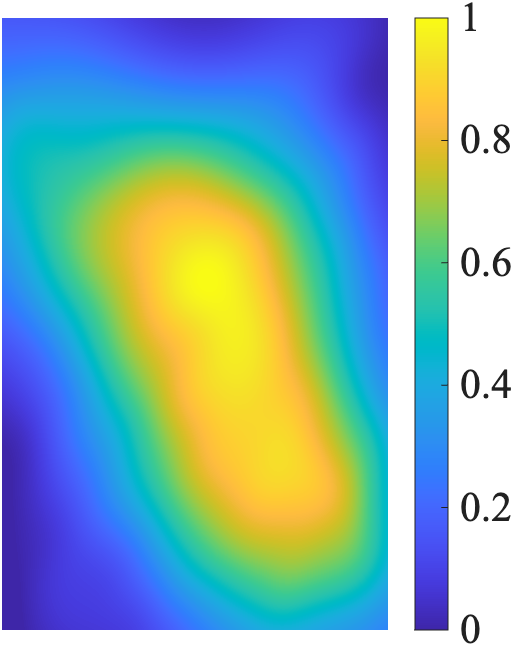} 
    & \includegraphics[width=.16\textwidth]{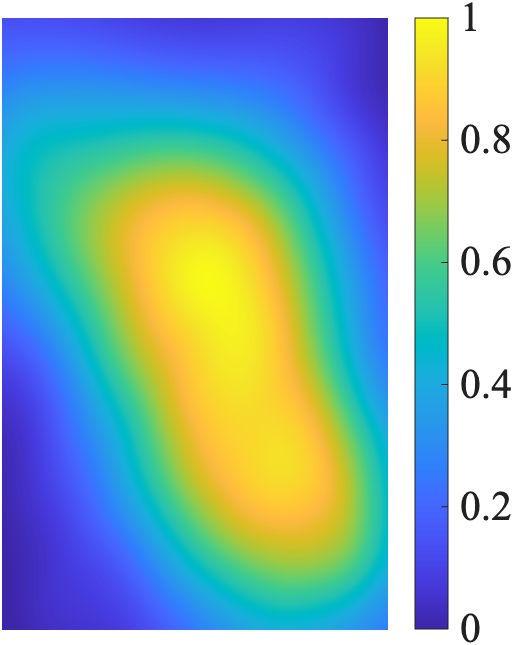}
    & \includegraphics[width=.16\textwidth]
    {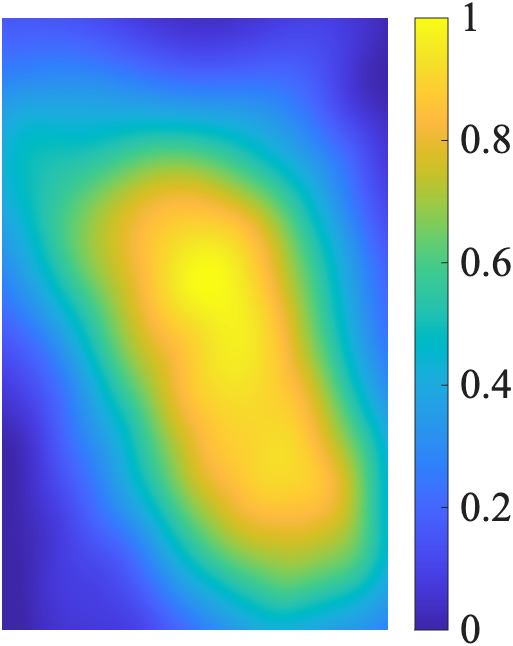}

\end{tabular}
\caption{Seismic inversion example. The ground truth is provided, along with reconstructions obtained using the initial guess $\bftheta_0$, the SAA approach, and the \mmmc approach with the maximum number of inner iterations identified in parentheses. }
\label{fig:SIrecons}
\end{figure}

For this numerical experiment, the SAA method replicates the setup in \cite{chung2024efficient} and serves as a comparison for the proposed \mmmc method. For the first test, we compare the SAA and \mmmc methods using approximations of the analytical gradients given in Option (a) of both methods. In \Cref{tab:SIresults} (a) it can be seen that the \mmmc approach has a faster runtime, uses fewer matvec operations, and achieves a smaller relative reconstruction error when the maximum inner iterations is set to 2.  Relative reconstruction errors are computed as $\|\widehat{\bfx} - \bfx_{\rm true}\|_2/\| \bfx_{\rm true} \|_2$ where $\widehat{\bfx}$ is the solution computed at the SAA estimate, $\widehat{\bftheta}_{\rm SAA}$, or the \mmmc estimate, $\widehat{\bftheta}_{\rm M^3C}$ using PCG. Although both methods produce comparable relative reconstruction errors for a larger number of maximum inner iterations, the SAA method outperforms \mmmc in runtime and number of matvecs. This indicates that the number of inner iterations used for \mmmc has an effect on the performance of the method and so should be chosen carefully. The final reconstructions for the SAA method and the \mmmc method using maximum inner iterations of 2 and 5 provided in \Cref{fig:SIrecons}, along with the reconstruction errors, show that both methods lead to good approximations of the true solution.

\begin{table}

\centering
\renewcommand{\arraystretch}{1.2}
\begin{tabular}{|c|c|c|c|c|c|c|c|c|} \cline{1-9}
  & \multicolumn{1}{c|}{} & max inner & total & total & time & \multicolumn{2}{c|}{matvecs} & rel \\ \cline{7-8}
   & \multicolumn{1}{c|}{} & iter  & iter & fn evals & (s) & $\bfA$ \& $\bfA^\top$ & $\bfQ$ & error  \\ \cline{1-9}
    \multirow{ 5}{*}{(a)}& SAA & - & 21 & 53 & 123 & 43,693 & 21,555 & 0.04848\\
    \cline{2-9} 
    & \multirow{ 4}{*}{\mmmc} &2 & 21 & 129 & 93 & 14,823 & 6,702 & 0.04269\\
    & &5 & 83 & 429 & 273 & 47,841 & 21,561 & 0.04863\\
    & &10 & 143 & 665 & 423 & 75,767 & 34,226 & 0.04864\\
    & &15 & 155 & 652 & 427 & 77,586 & 35,207 & 0.04864 \\  \hline
    \hline
    \multirow{ 5}{*}{(b)}& 
    SAA & - & 14 & 67 & 177 & 78,735 & 39,334 & 0.04409\\
    \cline{2-9} 
    & \multirow{ 4}{*}{\mmmc} & 2 & 18 & 183 & 58 & 18,477 & 9,147 & 0.03629\\
    & & 5 & 44 & 350 & 102 &  34,542 & 17,096 & 0.03710\\
    & & 10 & 133 & 1,142 & 297 & 109,788 & 54,323 & 0.04891\\
    & & 15 & 150 & 1,157 & 305 & 113,547 & 56,195 & 0.04864 \\  
    \hline
\end{tabular}

\captionof{table}{Seismic example. For SAA and \mmmc approaches with different maximum number of inner iterations, we compare the total number of iterations, total function evaluations, runtime, total number of matvecs, and relative image reconstruction errors. 
Results in Option (a) use approximations of the analytic gradients, and results in Option (b) use finite difference approximations.}
\label{tab:SIresults}

\end{table}

Next, we compare the SAA and \mmmc methods using a finite difference approximation of the gradients from Option (b) of both methods. The results provided in \Cref{tab:SIresults} (b) demonstrate for maximum inner iterations of 2 and 5, the \mmmc method has a faster runtime and uses fewer matvec operations. Additionally, these two setups achieve a smaller relative reconstruction error than the SAA method and both SAA and \mmmc using the Monte Carlo approximate gradients. Moreover, comparing the number of total function evaluations to the number of matvecs, it can be seen that the \mmmc objective function is computationally cheaper than the SAA objective function. For a larger number of maximum inner iterations, the SAA method has a faster runtime, fewer matvecs, and a smaller relative reconstruction error.

\subsection{Super-Resolution}\label{sec:sr}
Next, we consider a super-resolution image reconstruction problem where the goal is to construct a high-resolution image from multiple low-resolution images \cite{chung2006numerical,milanfar2017super}. The measurements are a collection of vectorized $32\times 32$ low-resolution images $\bfb^{(1)},\ldots, \bfb^{(8)}\in\bbR^{1,204}$ that contain information about the same object. In practice, observations can be obtained from multiple sensors aimed at the same object or by imaging an object at different time points. Specifically,
$
\bfb^{(i)} = \bfD\bfS(\bftheta^{(i)})\bfx + \bfe_i
$
where $\bfD$ transforms a high-resolution image into a low-resolution image and $\bfS(\bftheta^{(i)})$ is a sparse matrix that performs a geometric distortion on the high-resolution image $\bfx$ defined by unknown parameters\footnote{Note that for Case II, $\bftheta = \bfy$.} $\bftheta^{(i)}\in\bbR^6$ for $1 \le i \le 8$. For this numerical experiment, noisy observations are generated by distorting the true high-resolution image and performing a decimation. The collection of low-resolution measurements and the true high-resolution image can be found in \Cref{fig:SR_ref}.

\begin{figure}[t!]
    \centering
\begin{tabular}{c|cccc}
     True & \multicolumn{4}{c}{Low Resolution Images}  \\
     \includegraphics[width=.15\textwidth]{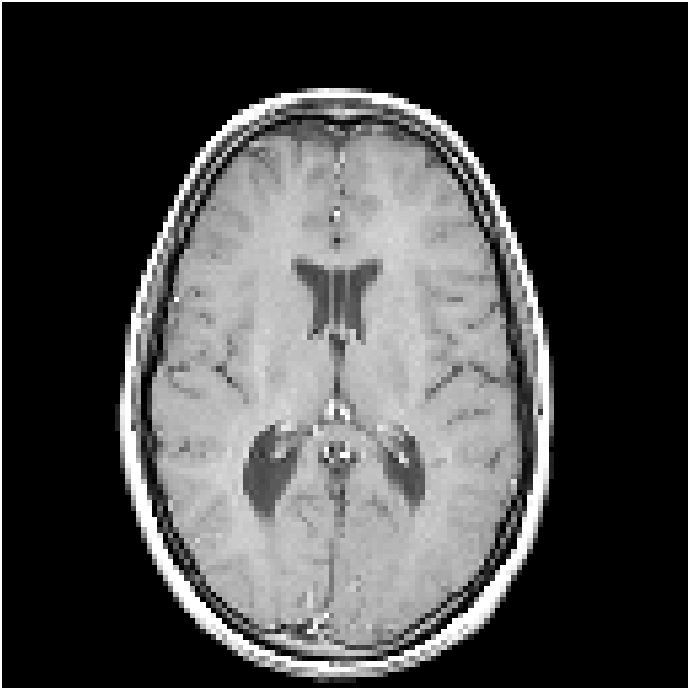}
     & \includegraphics[width=.15\textwidth]{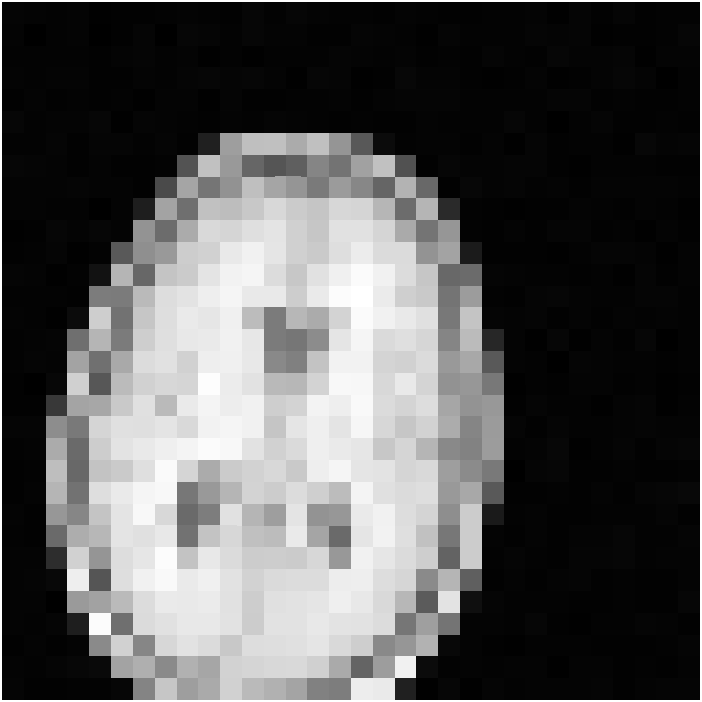}
     & \includegraphics[width=.15\textwidth]{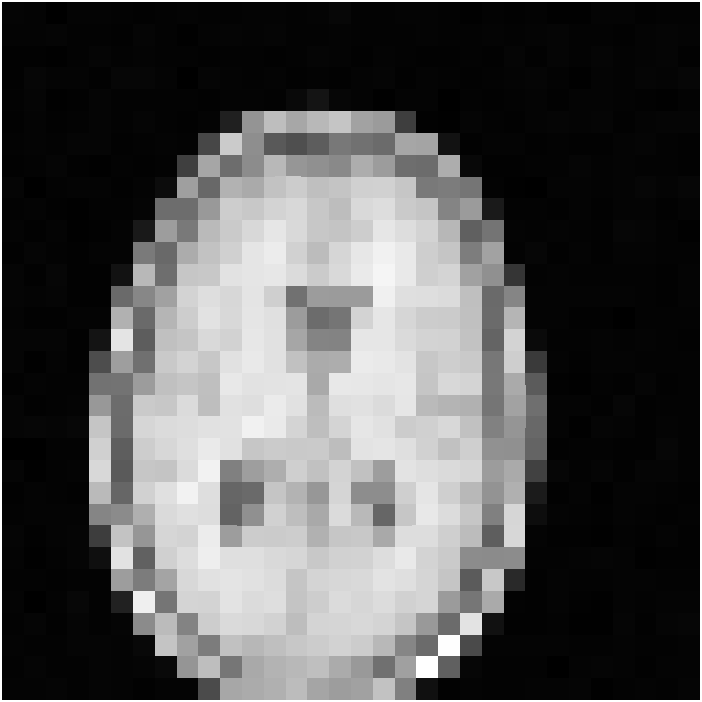}
     & \includegraphics[width=.15\textwidth]{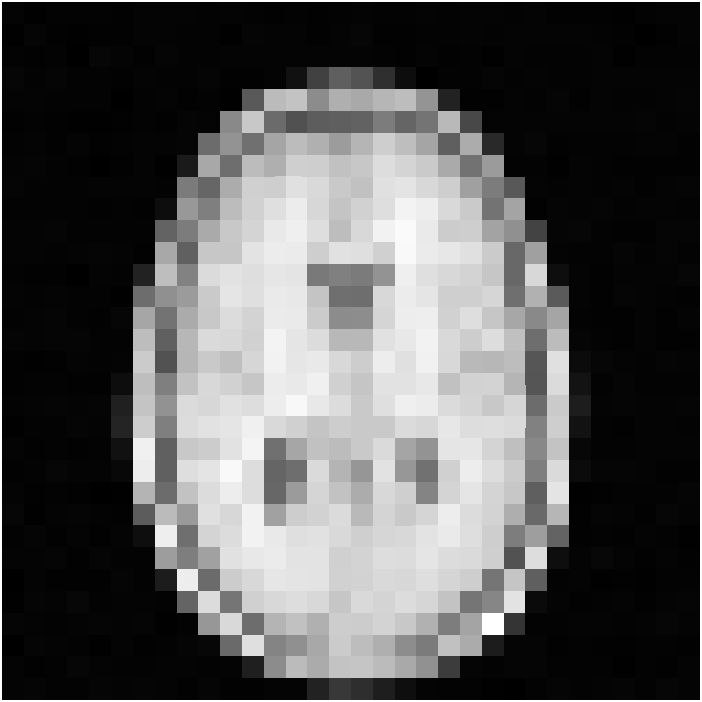}
     & \includegraphics[width=.15\textwidth]{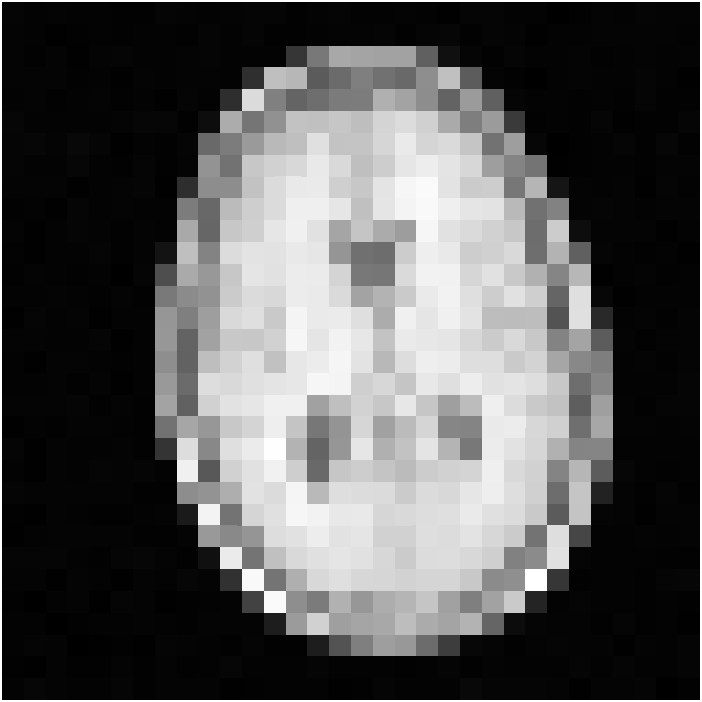} \\ 
     Reference & \multicolumn{4}{c}{} \\
     \includegraphics[width=.15\textwidth]{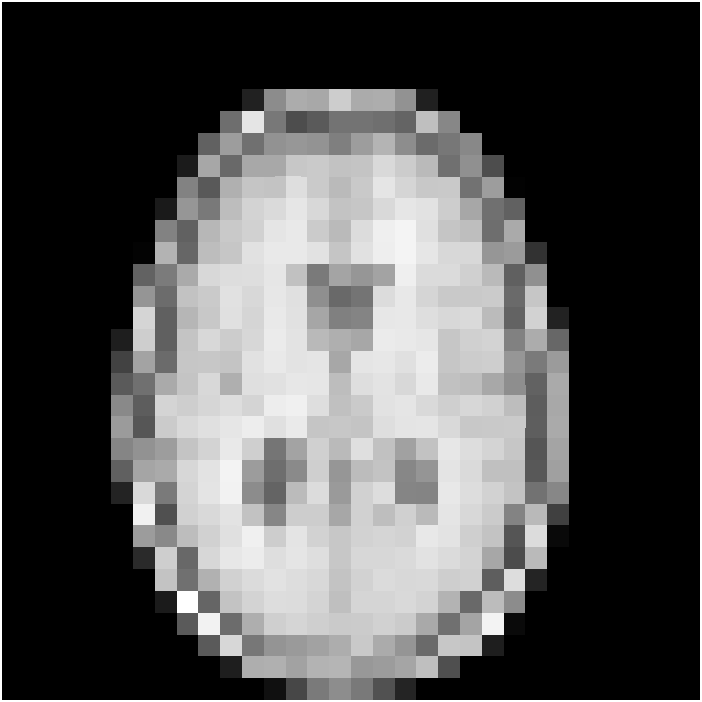}
     & \includegraphics[width=.15\textwidth]{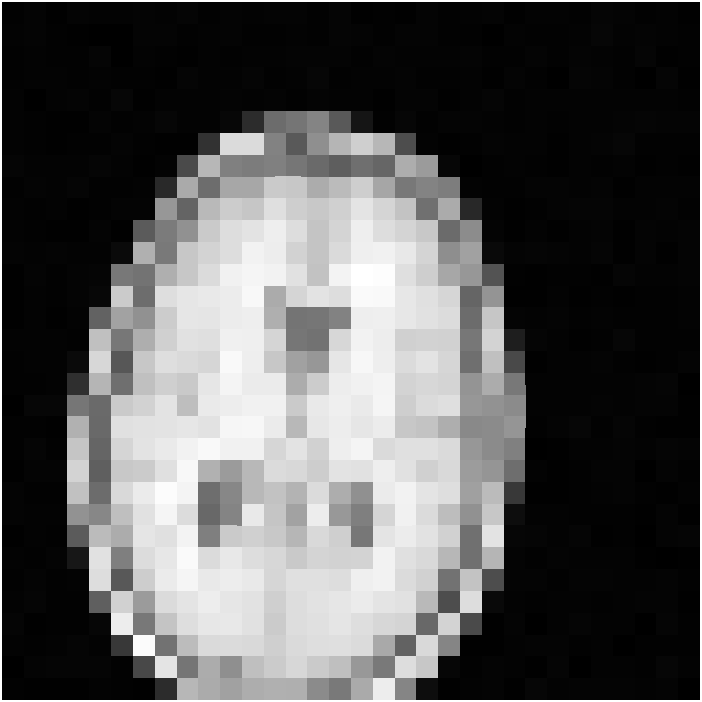}
     & \includegraphics[width=.15\textwidth]{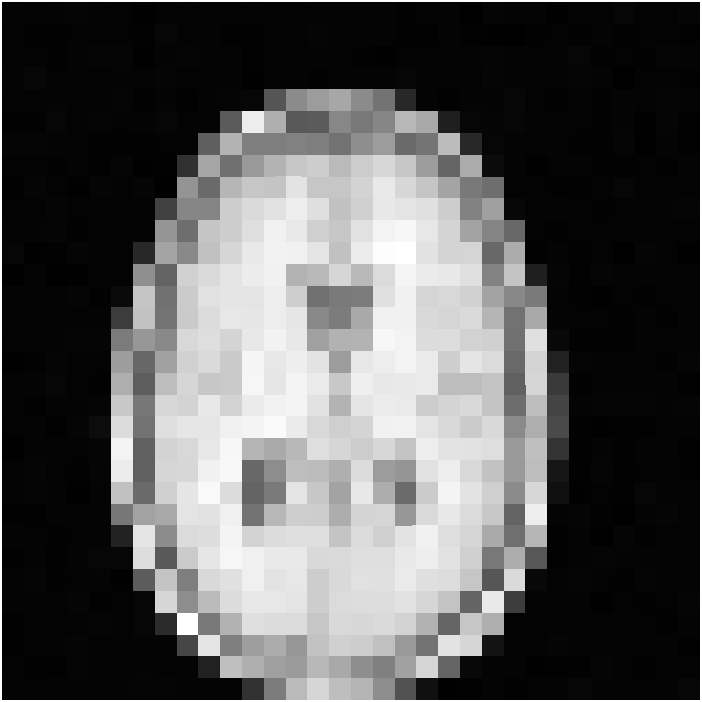}
     & \includegraphics[width=.15\textwidth]{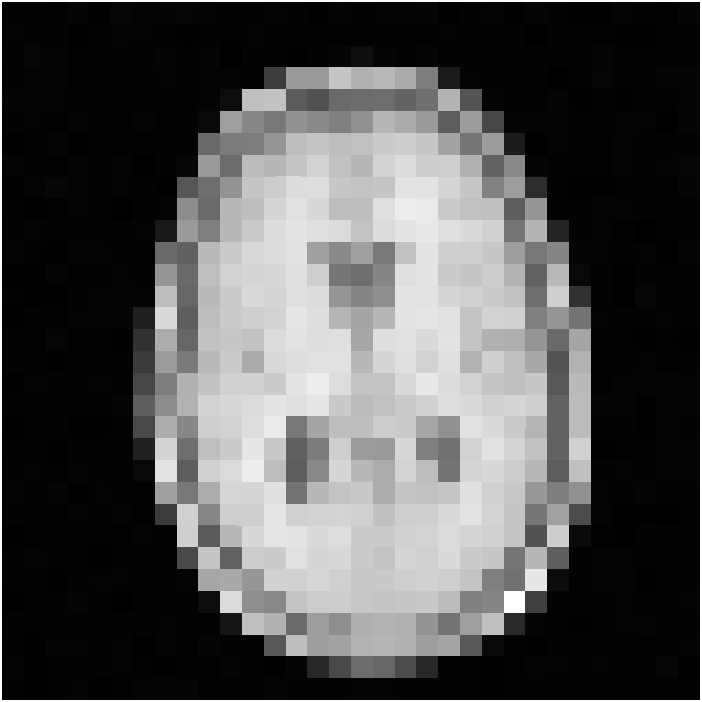}
     & \includegraphics[width=.15\textwidth]{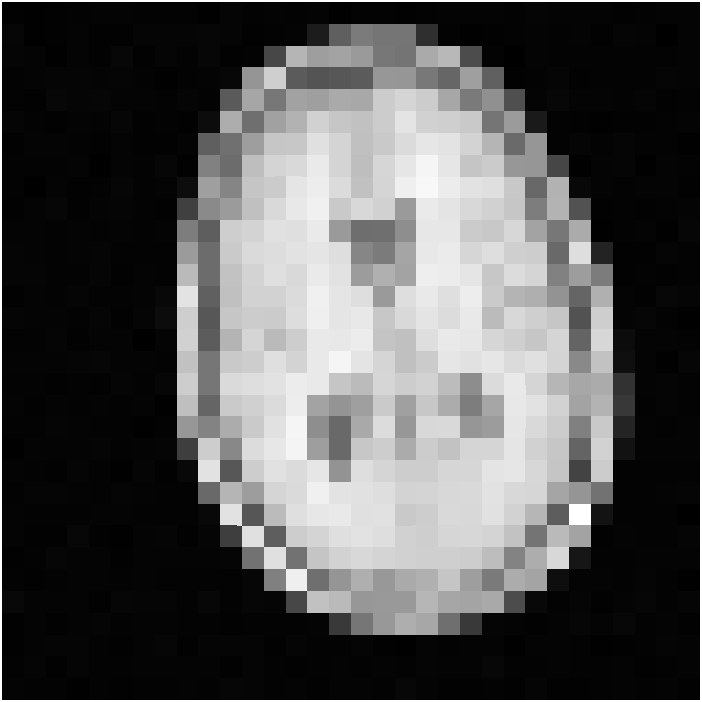} \\ 

\end{tabular}
    \caption{Super-resolution example. The true high resolution image is provided, along with multiple low resolution images, including the reference image.}
    \label{fig:SR_ref}
\end{figure}

The separable nonlinear inverse problem is then given by \cref{eq:case2}
where
\begin{equation*}
    \bfb=\begin{bmatrix}
        \bfb^r\\ \bfb^{(1)}\\ \vdots \\ \bfb^{(8)}
    \end{bmatrix}, \quad
    \bftheta=\begin{bmatrix}
        \bftheta^{(1)}\\ \vdots \\ \bftheta^{(8)}
    \end{bmatrix}, \quad
    \bfA(\bftheta) = \begin{bmatrix}
        \bfD\\ \bfD\bfS(\bftheta^{(1)}) \\ \vdots \\ \bfD\bfS(\bftheta^{(8)})
    \end{bmatrix},
\end{equation*}
$\bfx\in\bbR^{16,384}$ is the unknown $128\times 128$ high-resolution image, $\bfmu_\bfx = \bfzero$, $\bfQ_\bfx=25^2\bfI$, and $\bfR = (0.37)^2\bfI$. The reference image, denoted by $\bfb^r$, does not undergo a translation or rotation meaning $\bfD\bfx = \bfb^r$. Moreover, we use a Gaussian hyperprior $\bftheta\sim\mc{N}(\bf0,\bfI)$.
The initial guess $\bftheta_0$, provided in \Cref{fig:SR_results} alongside the true parameters $\bftheta_{\rm true}$, is calculated in a pre-registration phase and then scaled to be between the lower bound $l_b=-0.2$ and the upper bound $l_u = 0.2$. For both methods, we use the fmincon interior point method with box constraints defined by $l_b$ and $u_b$ on all components of $\bftheta$. For the SAA approach, the maximum number of iterations is 200 and the step tolerance is $1e-6$. For the \mmmc approach, the maximum number of inner iterations is 150 and the step tolerance is $1e-6$. For the outer-loop, the maximum number of iterations is 200 and the stopping criteria is $\|\widehat\bftheta_t-\widehat\bftheta_{t+1} \|_2/\| \widehat\bftheta_{t+1}\|_2 < 1e-6$. For each method, we use $N_t=100$ independent Rademacher vectors for the Monte Carlo trace estimation.

\begin{figure}
\renewcommand{\arraystretch}{1.3}
\centering
\begin{tabular}{cccc}
         \multicolumn{2}{c}{$\bftheta_{\rm true}$} & \multicolumn{2}{c}{$\bftheta_0$} \\ 
     \includegraphics[width=.16\textwidth]{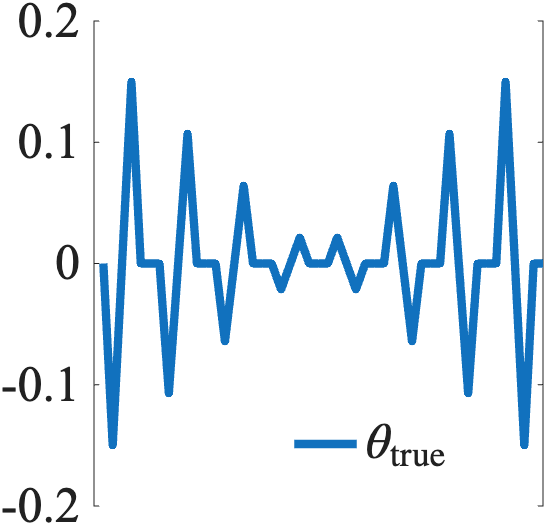}
     &
     \includegraphics[width=.195\textwidth]{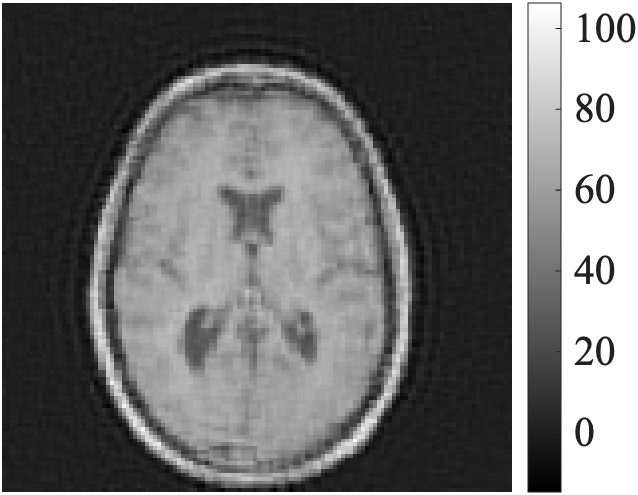}
     & 
     \includegraphics[width=.16\textwidth]{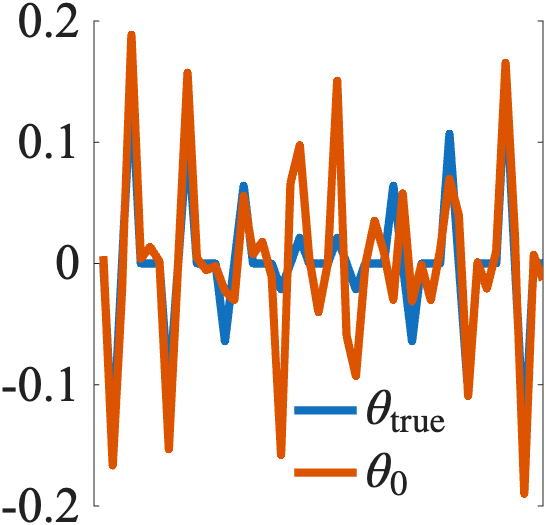}
     & 
     \includegraphics[width=.195\textwidth]{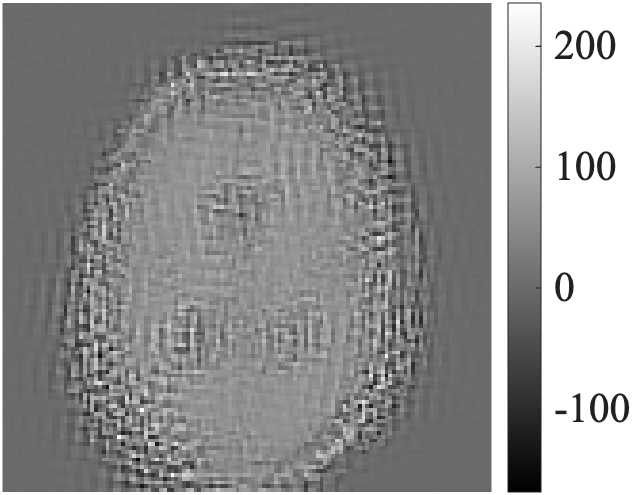}
    \\
    \multicolumn{2}{c}{\textbf{SAA}} & \multicolumn{2}{c}{\textbf{\mmmc}}
    \\ 
    \includegraphics[width=.16\textwidth]{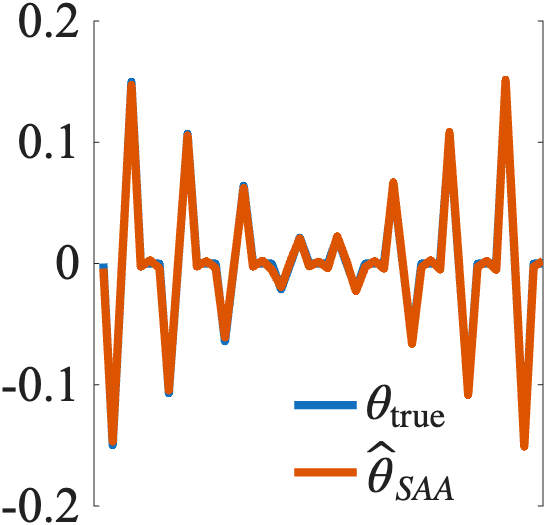} 
     &
    \includegraphics[width=.195\textwidth]{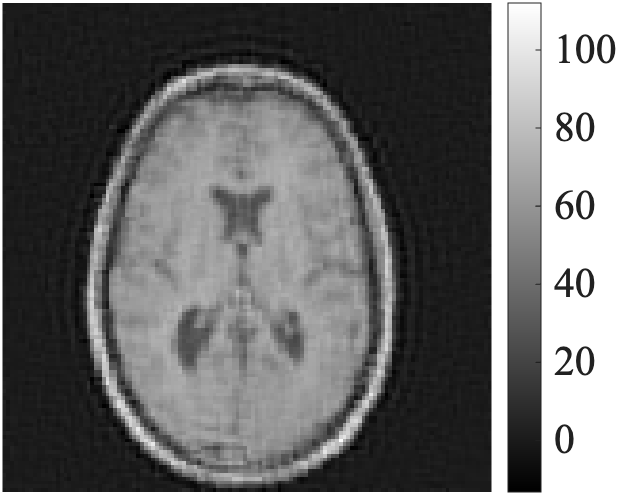}
    & 
     \includegraphics[width=.16\textwidth]{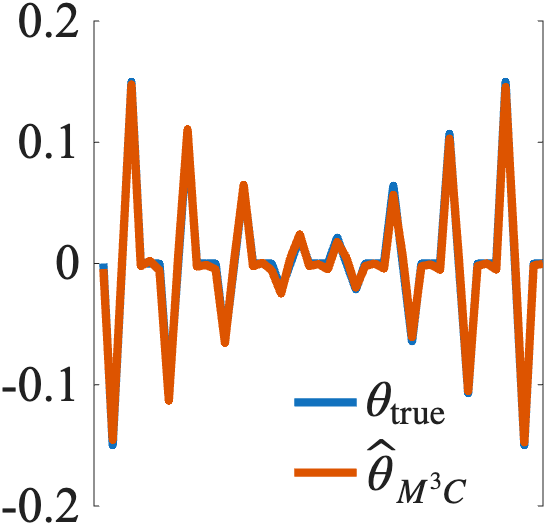}
     &  \includegraphics[width=.195\textwidth]{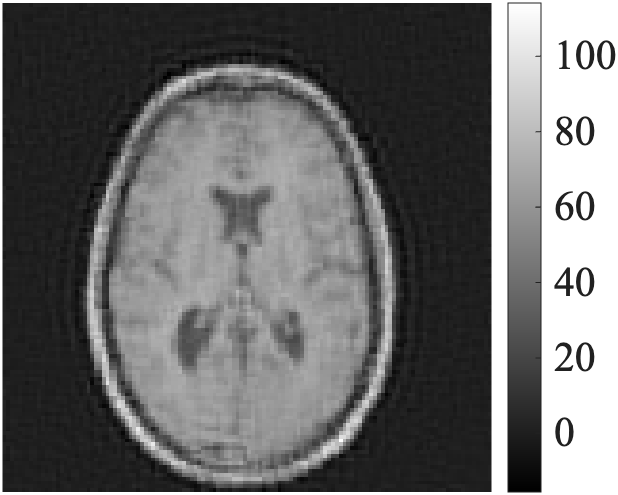}
     \\
\end{tabular}
\caption{Super-resolution example. Top row: The true and initial parameters, $\bftheta_{\rm true}$ and $\bftheta_0$, along with the resulting image reconstructions. Bottom row: Parameter estimations, $ \widehat{\bftheta}_{\rm SAA}$ and $\widehat{\bftheta}_{\rm M^3C}$, computed using Option (a), along with the resulting image reconstructions. The relative errors of the estimated parameters, computed as $\|\widehat{\bftheta}-\bftheta_{\rm true}\|/\|\bftheta_{\rm true}\|$, are $0.0473$ for $ \widehat{\bftheta}_{\rm SAA}$ and $0.0656$ for $\widehat{\bftheta}_{\rm M^3C}$. Additional quantitative comparisons are given in Table~\ref{tab:SR_approx}.}
\label{fig:SR_results}
\end{figure}

\begin{table}[ht!]
    \centering
    \renewcommand{\arraystretch}{1.25}
    \begin{tabular}{|c|c|c|c|c|c|c|c|}
    \hline
     & & total & total & time & \multicolumn{2}{c|}{matvecs}  & rel
    \\
    \cline{6-7}
     & & iter & fn evals  & (s) & $\bfA$ \& $\bfA^\top$ & $\bfQ$ & error 
    \\ \hline
    \multirow{ 2}{*}{(a)}& \textbf{SAA} & 164 & 703 & $5.73\times 10^3$  & 28,582,364 & 14,189,247 & 0.1866
    \\ 
    \cline{2-8}
    & \textbf{\mmmc} & 172 & 791 & $5.31\times 10^3$ & 15,678,192 & 7,724,401  & 0.1876
    \\ \hline \hline
    \multirow{ 2}{*}{(b)}& \textbf{SAA} & 173 & 9,047  & $9.32\times 10^3$   & 53,459,075  & 26,725,014 & 0.1865
    \\ 
    \cline{2-8}
    &  \textbf{\mmmc} & 220 & 11,616  & $561$ & 3,751,594 & 1,869,989 & 0.1868
    \\ \hline
    \end{tabular}
    
    \caption{Super-resolution example. A comparison of total number of iterations (total number of inner iterations for \mmmc), total function evaluations, runtime, total number of matvecs, and relative reconstruction errors for SAA and \mmmc.
    Results in Option (a) use approximations of the analytic gradients, and results in Option (b) use finite difference approximations.
    }
    \label{tab:SR_approx}
\end{table}

We compare the two methods using approximate gradients in Option (a) and Option (b) for SAA and \mmmc. The results provided in \Cref{tab:SR_approx}(a) show that the methods are comparable in terms of runtime and relative reconstruction error along with the parameter estimations and reconstructions provided in \Cref{fig:SR_results}. However, the \mmmc method requires fewer matvecs despite using more function evaluations, which implies that \mmmc has a lower computational cost than SAA. 
From the results provided in \Cref{tab:SR_approx}(b), we observe that the \mmmc method results in a 16 times speedup compared to the SAA method along with a significant decrease in the number of matvec operations while achieving a similar relative reconstruction error. This indicates that evaluating $\widehat{\mc{G}}_{N_t}(\bftheta\mid\bftheta_t)$ is computationally cheaper than evaluating $\widehat{\mc{F}}_{\rm SL}(\bftheta)$. Comparing the results using different gradient approximations, we conclude that the main computational cost of the \mmmc method with Option (a) is constructing a Monte Carlo approximation of the gradient, as the finite difference approach (Option (b)) uses $\approx4.3$ times fewer matvecs but $\approx 1.4$ times more function evaluations.

\begin{figure}[t]
    \centering
    \small 
    
    \begin{subfigure}[t]{0.24\textwidth}
        \centering
        \includegraphics[width=\linewidth]{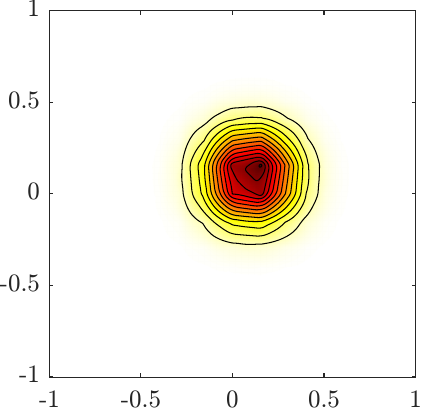}
        \caption{Prior mean $\bfmu_\bfx$}
    \end{subfigure}
    \hfill
    \begin{subfigure}[t]{0.24\textwidth}
        \centering
        \includegraphics[width=\linewidth]{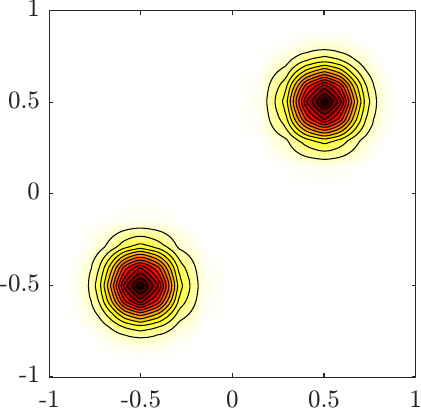}
        \caption{True $u_0(\bfr;\bftheta_{\mathrm{true}})$}
    \end{subfigure}
    \hfill
    \begin{subfigure}[t]{0.24\textwidth}
        \centering
        \includegraphics[width=\linewidth]{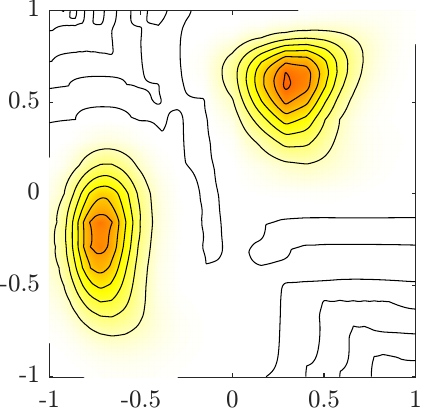} 
        \caption{True $u(\bfr, T;\bftheta_{\mathrm{true}})$}
    \end{subfigure}
    \hfill
    \begin{minipage}[b]{0.24\textwidth}
        \centering
        \vspace{-3cm}  
    \end{minipage}

    \vspace{-0.15cm}  

    \begin{subfigure}[t]{0.24\textwidth}
        \centering
        \includegraphics[width=\linewidth]{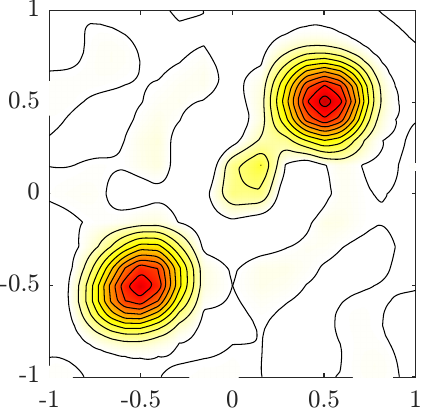}
        \caption{$u_0(\bfr;\widehat{\bftheta}_{\mathrm{M^3C}})$}
    \end{subfigure}
    \hfill
    \begin{subfigure}[t]{0.24\textwidth}
        \centering
        \includegraphics[width=\linewidth]{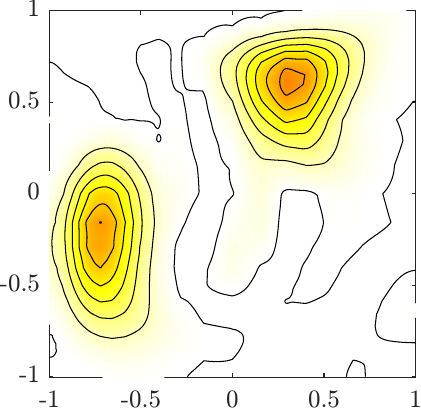}
        \caption{$u(\bfr,T;\widehat{\bftheta}_{\mathrm{M^3C}})$}
    \end{subfigure}
    \hfill
    \begin{subfigure}[t]{0.24\textwidth}
        \centering
        \includegraphics[width=\linewidth]{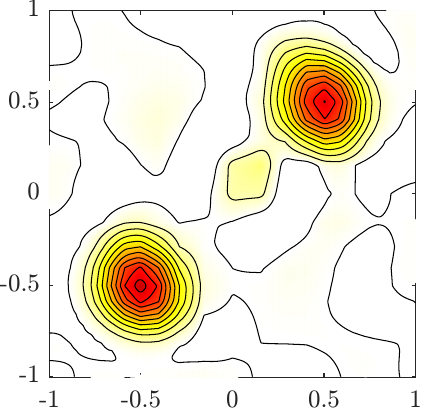}
        \caption{$u_0(\bfr;\bftheta_{\mathrm{pr}})$}
    \end{subfigure}
    \hfill
    \begin{subfigure}[t]{0.24\textwidth}
        \centering
        \includegraphics[width=\linewidth]{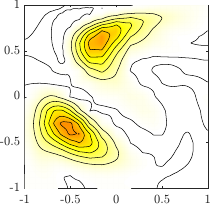}
        \caption{$u(\bfr,T;\bftheta_{\mathrm{pr}})$}
    \end{subfigure}
\vspace{-0.15cm}  
 \caption{Advection--diffusion example. The top row provides the prior mean $\bfmu_\bfx$, along with the initial and final concentration fields corresponding to ground truth parameters $\bftheta_{\rm true}$. 
 In the bottom row, we provide the reconstructed initial concentrations and the corresponding predicted final states for the \mmmc estimate $\widehat{\bftheta}_{\rm M^3C}$ and the prior $\bftheta_{\rm pr}$.}
    \label{fig:inverse_results_grid}
\end{figure}

\subsection{Advection--Diffusion with an Uncertain Navier--Stokes Velocity}
\label{sec:ns}
We consider a model problem from contaminant source identification, which involves the transport of a passive scalar concentration \(u = u(\bfr,\tau)\) over a spatial domain
\(\Omega \subset \mathbb{R}^2\) and a time interval \((0,T]\).
The evolution of \(u\) is governed by an advection--diffusion equation driven by a velocity field
\(\bm{v}_\theta\), which is obtained as the solution of a Navier--Stokes model with an uncertain forcing term
\(\bm{f}_\theta\). 
For our numerical experiments, we utilize the  IFISS library \cite{IFISS}.

We consider the forward problem mapping the unknown initial concentration $u_0\in L^2(\Omega)$, which is the primary parameter of interest, to observations. For a constant diffusion $\kappa > 0$ and velocity $\bm{v}_\theta$, the state $u$ satisfies:
%
\begin{align}
\begin{aligned}
\frac{\partial u}{\partial \tau} - \nabla \cdot (\kappa \nabla u) + \bm{v}_\theta \cdot \nabla u &= 0 
&& \text{in } \Omega_T := \Omega \times (0,T), \\
u(\cdot,0) &= u_0 
&& \text{in } \Omega,
\end{aligned}\label{eq:advection_diffusion}
\end{align}
subject to standard Dirichlet and Neumann boundary conditions on $\partial\Omega \times (0,T)$.  

Observations are collected at \(m\) spatial sensor locations \(\{\bfr_1,\dots,\bfr_m\}\subset\Omega\) at the final time only ($\tau=T$). The inverse problem involves reconstructing the initial condition $u_0$, represented by the discrete vector $\bfx$, from the final time measurements $\{u(\bfr_j, T)\}_{j=1}^{m}$.  Denote by \(\mathcal{H}:\;L^2(\Omega)\to\mathbb{R}^{m}\) the pointwise sampling operator
\[ 
\bfA(\bftheta) :=\mathcal{H}(u) = \big(u(\bfr_1,T;\bftheta),\dots,u(\bfr_{m},T;\bftheta)\big)^\top,
\]
with~\eqref{eq:case2} as the data model with prior mean $\bfmu_\bfx$ (provided in Figure \ref{fig:inverse_results_grid} (a)) and Mat\'ern covariance matrix $\bfQ_\bfx$ with correlation length \(0.2\) and smoothness parameter \(0.5\). 
Synthetic observations are generated on a $15 \times 15$ uniform grid with $4\%$ relative additive noise. We set $\Omega = [-1, 1]^2$ and $T = 1.5$.
The resulting initial and final concentration fields, corresponding to the ground truth $\bftheta_{\rm true}$, are provided in Figures \ref{fig:inverse_results_grid} (b) and (c) respectively.




The velocity field \(\bm{v}_\theta\) is obtained from the steady incompressible
Navier--Stokes equations posed on \(\Omega \subset \mathbb{R}^2\):
%
\begin{equation}\label{eq:stokes-momentum}
\begin{aligned}
  -\nu\nabla^2 \bm{v}_\theta 
  + \bm{v}_\theta \cdot \nabla\bm{v}_\theta + \nabla p
  &= \bm{f}_\theta, \\
  \nabla \cdot \bm{v}_\theta &= 0,
\end{aligned}
\qquad \text{in } \Omega,
\end{equation}
with boundary condition \(\bm{v}_\theta = \mathbf{g}\) on \(\partial\Omega\) and
pressure normalization \(p(\mathbf{r}_0)=0\) for some \(\mathbf{r}_0 \in \partial\Omega\).
Here \(\bm{v}_\theta : \Omega \to \mathbb{R}^2\) denotes the velocity field,
\(p : \Omega \to \mathbb{R}\) the pressure, and $\nu>0$ is the kinematic viscosity. 
%
The forcing term \(\bm{f}_\theta \in [L^2(\Omega)]^2\) models spatially localized body forces with uncertain amplitudes:
\begin{align}
\bm{f}_\theta(\bm{r})
=
\biggl(
\sum_{j=1}^\ell \theta^{(j)} \exp(-\|\bm{r}-\bm{r}_{c_j}\|^2 / r^2),\, 0
\biggr).
\label{eq:forcing-term}
\end{align}
where \(\bm{r}_{c_j} \in \Omega\) are fixed and \(r>0\) is a prescribed length scale.
The parameter vector \(\bftheta = (\theta^{(1)},\dots,\theta^{(\ell)})^\top \in \mathbb{R}^\ell\) represents the unknown forcing magnitudes.
Homogeneous Dirichlet boundary conditions are imposed on all boundaries except the top lid, where a unit tangential velocity directed clockwise is prescribed. Due to  the strong nonlinearity of the Navier--Stokes equations, the induced parameter-to-observable map is highly nonlinear. In particular, both the magnitude and spatial placement of the forcing terms substantially influence the resulting flow field and downstream quantities of interest. Consequently, the inverse problem considered here  exhibits pronounced nonlinear coupling and parameter interaction effects.


\begin{figure}[t!]
\centering
\includegraphics[width=0.8\textwidth, height=4.5cm]{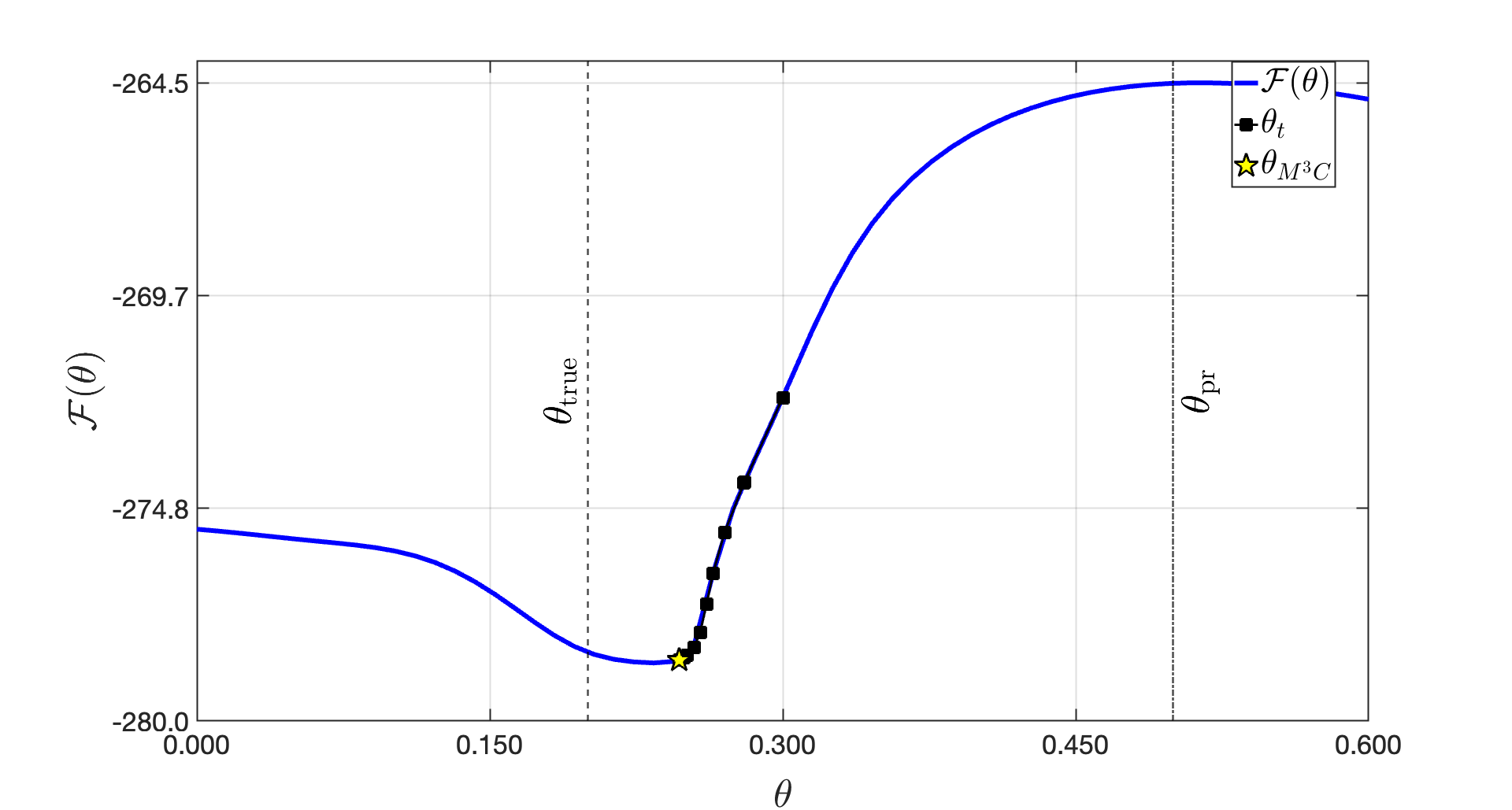}
\caption{Advection--diffusion example. In the case $\ell=1$, we provide objective function values $\calF(\theta)$ corresponding to the \mmmc iterates $\theta_t$ and the final \mmmc estimate $\theta_{M^3C}$.  The initial guess was $\theta_0 = 0.3$, and the vertical markers indicate the prior mean $\theta_{\mathrm{pr}}$ and the the true parameter $\theta_{\mathrm{true}}$.}
\label{fig:psi_k1}
\end{figure}

In our numerical experiments, we use {$\nu=1/200$} and forcing term \eqref{eq:forcing-term} with $\ell=1$, $\bfr_{c_1} = (-0.4, 0.3)^\top $, and $r=1$. We adopt a Gaussian prior for the source amplitudes $ \bftheta = \bmat{\theta}$, where
  $\theta \sim \mc{N}(\theta_{\rm pr}, 0.5)$ with $\theta_{\rm pr} = 0.5$. Figure~\ref{fig:psi_k1} illustrates the behavior of the function values $\calF(\theta)$ at the \mmmc iterates $\theta_t$, where for optimization, we use the initial guess \(\theta_0 = 0.3\). Function values at the prior mean $\calF(\theta_{\rm pr})$ and the true parameter $\calF(0.2)$ are provided for reference.
This example highlights the nonconvex structure induced by the nonlinear forward model and the discrepancy between prior information and data-driven  estimates. 
The \mmmc framework utilizes $10$ outer iterations, with the Monte Carlo approximation in \eqref{eqn:surrmc} computed using $N_t = 16$ samples and $n = 30$ probing vectors for the randomized preconditioner \cite{Frangella}
used in PCG. 
This configuration allows us to leverage parallel matvecs.
In this setting, the optimizer (inner iterations) required an average of $4.7$ iterations per outer step.
The initial and final concentration fields corresponding to the final \mmmc iterate $\widehat{\bftheta}_{\rm M^3C}$ and the prior  $\bftheta_{\rm pr}$
are provided in the bottom row of Figure \ref{fig:inverse_results_grid}. We do not include results from SAA, which had trouble converging for this problem. 

\section{Conclusions}
\label{sec:conclusions}
In this work, we develop the \mmmc method, which constructs a sequence of MM majorants that, with a Monte Carlo estimator, results in an efficient computational method for hyperparameter estimation.  Specifically, we adopt an MM approach where a sequence of majorants is defined for the log-determinant term, resulting in an inner-outer optimization scheme.  To handle expensive trace terms in the majorant, Monte Carlo estimators are used.  A probabilistic analysis of MM approaches with inexactness shows that \mmmc iterates converge with high probability to a critical point of the original cost functional.  Also, we derive bounds on the minimal number of samples for the trace estimator and bounds on the number of Monte Carlo samples and degree of the Lanczos polynomial to guarantee a small backward error.
We show how these approximations can be used for hyperparameter estimation in hierarchical Bayesian inverse problems, and we extend these methods to separable nonlinear forward models by treating the unknown model parameters as hyperparameters.


\appendix
\bibliographystyle{abbrv}
\bibliography{references}
\end{document}


\nolinenumbers
\maketitle

This report provides supplementary material for the paper “A Majorization-Minimization with Monte Carlo Approach for Hyperparameter Estimation”. In \Cref{sec:SAA} we provide additional details regarding gradient computations and computational costs for the SAA approach.  Then, in \Cref{sec:proofs}, we provide proofs of results used in \Cref{main-sec:analysis}.  Finally, in \Cref{sec:majorant}, we provide a visualization of the majorant for a small seismic example problem.


\section{Computational details for SAA}
\label{sec:SAA}
\paragraph{Gradient computations}
Consider the original objective function $\mc{F}(\bftheta)$ given in \eqref{main-eq:F}, where analytical expressions for the gradients are given in \eqref{main-eq:fullGradpsi} corresponding to Case I and \eqref{main-eq:fullGrady} corresponding to Case II.  For Case I, Monte Carlo approximations of the gradients of $\mc{F}$ were developed and used for SAA in \cite{chung2024efficient}.  These gradient approximations were used for \mmmc, see \Cref{main-sec:NumEx}. 

Here, we describe similar Monte Carlo approximations of the gradient evaluations for Case II. The derivatives are provided in \eqref{main-eq:fullGrady} and are copied here for completeness,
\begin{align}
        & \frac{\partial{\mc{F}}}{\partial y_j} =  -\frac{1}{\pi_{\bfy}(\bfy)} \frac{\partial\pi_{\bfy}(\bfy)}{\partial y_j} + \frac{1}{2}\trace\left(\bfPsi(\bftheta)^{-1}\frac{\partial \bfPsi(\bftheta)}{\partial y_j}\right) \label{eq:fullGrady}   \\
         &- \frac{1}{2}\Big[ \bfPsi(\bftheta)^{-1}\left(\bfA(\bfy)\bfmu_\bfx-\bfb\right)\Big]\t\left[\frac{\partial \bfPsi(\bftheta)}{\partial y_j}\bfPsi(\bftheta)^{-1}(\bfA(\bfy)\bfmu_\bfx-\bfb)-2\frac{\partial\bfA(\bfy)}{\partial y_j}\bfmu_x  \right].  \notag
\end{align}



First, notice that the last term in \Cref{eq:fullGrady} can be evaluated by reusing computations made during the objective function evaluation. Recall that we have already computed the solution to $\bfPsi(\bftheta) \bfr = \bfA(\bfy)\bfmu_\bfx-\bfb$.  Thus, we can evaluate the expression efficiently (for $1\leq j \leq \ell$), 
\begin{align*}
    \Big[ &\bfPsi(\bftheta)^{-1}\left(\bfA(\bfy)\bfmu_\bfx-\bfb\right)\Big]\t\left[\frac{\partial \bfPsi(\bftheta)}{\partial y_j}\bfPsi(\bftheta)^{-1}(\bfA(\bfy)\bfmu_\bfx-\bfb)-2\frac{\partial\bfA(\bfy)}{\partial y_j}\bfmu_x  \right]\\
    & = \bfr\t \left( \frac{\partial \bfPsi(\bftheta)}{\partial y_j} \bfr -2\frac{\partial\bfA(\bfy)}{\partial y_j}\bfmu_x\right), 
\end{align*}
where
\begin{equation*}
    \frac{\partial\bfPsi(\bftheta)}{\partial y_j}\bfr = \bfA(\bfy)\bfQ_\bfx\frac{\partial \bfA(\bfy)\t}{\partial y_j} \bfr + \frac{\partial\bfA(\bfy)}{\partial y_j}\bfQ_\bfx\bfA(\bfy)\t\bfr.
\end{equation*}

Next, we compute an approximation of the trace term, using a Monte Carlo approximation, as done in \Cref{main-e:hutch}. Given $N$ independent Rademacher vectors $\bfw_i \in \{\pm1\}^m$, the Hutchinson estimator for the trace is given by
\begin{equation*}
\trace\left( \bfPsi(\bftheta)^{-1}\frac{\partial \bfPsi(\bftheta)}{\partial y_j} \right) \approx \frac{1}{N_t} \sum_{i=1}^{N_t} \bfw_i^\top \bfPsi(\bftheta)^{-1}\frac{\partial \bfPsi(\bftheta)}{\partial y_j}\, \bfw_i .
\end{equation*}   
We can also symmetrize the estimator using the cyclic property of the trace,
\begin{equation*}
\trace\left( \bfPsi(\bftheta)^{-1/2}\frac{\partial \bfPsi(\bftheta)}{\partial y_j} \bfPsi(\bftheta)^{-1/2}\right) \approx \frac{1}{N_t} \sum_{i=1}^{N_t} \bfw_i^\top \bfPsi(\bftheta)^{-1/2}\frac{\partial \bfPsi(\bftheta)}{\partial y_j}\bfPsi(\bftheta)^{-1/2}\, \bfw_i.
\end{equation*}
Additionally, trace estimation can be efficiently computed using a preconditioner. Let $\bfG$ be a preconditioner that is easy to invert and satisfies $\bfG\t\bfG \approx \bfPsi(\bftheta)^{-1}$. 
Then, consider the factorization of $\bfPsi(\bftheta)^{-1}$,
\begin{equation*}
    \bfPsi(\bftheta)^{-1} = \bfG^{-\top}\left( \bfG\bfPsi(\bftheta)\bfG\t \right)^{-1/2}\left( \bfG\bfPsi(\bftheta)\bfG\t \right)^{-1/2}\bfG^{-1}.
\end{equation*}
Using the cyclic property of the trace leads to the Hutchinson trace estimate,
\begin{equation*}
\trace\left(\left( \bfG\bfPsi(\bftheta)\bfG\t \right)^{-1/2}\bfG^{-1}\frac{\partial \bfPsi(\bftheta)}{\partial y_j}\bfG^{-\top}\left( \bfG\bfPsi(\bftheta)\bfG\t \right)^{-1/2}\right) \approx \frac{1}{N_t}\sum_{i=1}^{N_t}\bfzeta_i\t\frac{\partial \bfPsi(\bftheta)}{\partial y_j}\bfzeta_i
\end{equation*}
for $1\leq j \leq \ell$, where $\bfzeta_i \equiv \bfG^{-\top}\left( \bfG\bfPsi(\bftheta)\bfG\t \right)^{-1/2}\bfw_i$ for $1\leq i \leq N_t$. Matvec computations with the inverse square root matrices can be done through a Lanczos method. Moreover, we can reuse the basis vectors from the Lanczos method used to approximate the log-determinant to efficiently compute $\bfzeta_i$ \cite{chung2024efficient}.
Thus, we get an approximation of $\nabla_\bfy\mc{F}$ with elements
\begin{equation}
    \frac{1}{\pi_{\bftheta}(\bftheta)}\frac{\partial\log\pi_{\bftheta}(\bftheta)}{\partial y_j} + \frac{1}{2N_t}\sum_{i=1}^{N_t}\bfzeta_i\t\frac{\partial \bfPsi(\bftheta)}{\partial y_j}\bfzeta_i - \frac{1}{2} \bfr\t\left( \frac{\partial \bfPsi(\bftheta)}{\partial y_j} \bfr -2\frac{\partial\bfA(\bfy)}{\partial y_j}\bfmu_x\right). \notag
\end{equation}

Moreover, we can use a finite difference approximation to estimate the partials $\frac{\partial \bfPsi(\bftheta)}{\partial y_j}\bfzeta_i$, $\frac{\partial \bfPsi(\bftheta)}{\partial y_j}\bfr$ and 
$\frac{\partial\bfA(\bfy)}{\partial y_j}\bfmu_x$ and a Preconditioned Conjugate Gradient (PCG) method is used for matvec operations with $\bfPsi(\bftheta)^{-1}$. It should be noted that we are not computing the exact derivatives of $\widehat{\mc{F}}_{\rm SL}$ or $\mathcal{F}$. 

 

\paragraph{Computational costs for SAA} 
Following the notation in Section \ref{main-ssec:compconsider}, here we provide a discussion of the computational costs for SAA:

\begin{enumerate}
    \item \textbf{Precomputation}:
With these assumptions, the cost of the precomputation involved in constructing the preconditioner is $\mc{O}(mr^2) + T_Ar$ flops.
\item \textbf{Objective function}: We can compute the preconditioner $\bfG$ in $\mc{O}(mr^2 + r^3)$ flops. The basis $\bfV_k$ can be computed in $kT_\Psi + \mc{O}(mk^2)$ flops, when Lanczos is used with reorthogonalization. In total, the log-determinant can be computed in $  N_t( kT_\Psi   +  \mc{O}(mk^2+k^3))$ flops. There is an additional cost of computing $\bfz$ and $\bfz\t\bfr$, which is $kT_\Psi + \mc{O}(mk)$ flops.

\item \textbf{Gradient computation}: Here, we consider taking the gradient with respect to $y_j$ for $1\leq j\leq \ell$ where computations with $\frac{\partial \bfPsi(\bftheta)}{\partial y_j}$, $\frac{\partial \bfPsi(\bftheta)}{\partial y_j}$, and $\frac{\partial \bfA(\bfy)}{\partial y_j}$ are approximated using finite difference. Using the precomputed $\bfzeta_t$, we can approximate $\sum_{i=1}^{N_t}\bfzeta_i\t\frac{\partial \bfPsi(\bftheta)}{\partial y_j}\bfzeta_i$ in $\ell N_t(2T_\Psi + \mc{O}(m))$ flops. Reusing $\bfr$ from above, $\bfr\t\frac{\partial \bfPsi(\bftheta)}{\partial y_j}\bfr$ can be approximated in $\ell(2T_\Psi + \mc{O}(m))$ and $\bfr\t\frac{\partial \bfA(\bfy)}{\partial y_j}\bfmu_\bfx$ can be approximated in $\ell(2T_A + \mc{O}(m))$.
\end{enumerate}

\section{Proofs} \label{sec:proofs}

In this section, we aim to give a proof of Lemma~\ref{main-cor:LsmoothnessG}. We will need the following lemma.
\begin{lemma}[Lipschitz continuity of the log-determinant MM surrogate]
\label{lem:lipschitz_Q}

For a fixed $\bftheta_t \in \Theta$, the function $\mc{Q}(\cdot \mid \bftheta_t)$  defined in~\eqref{main-eq:logdetmajorization} is Lipschitz continuous on $\Theta$. More precisely, for any $\bftheta, \bftheta' \in \Theta$:
\[
|\mc{Q}(\bftheta \mid \bftheta_t) - \mc{Q}(\bftheta' \mid \bftheta_t)|
\le
L \|\bftheta - \bftheta'\|,
\]
where the Lipschitz constant is given by $\displaystyle L = \frac{{m}}{\alpha} L_\Psi$.
\end{lemma}

\begin{proof}
Let $\bfW_t := \bfPsi^{-1}(\bftheta_t)$. Since the terms
$\log\det \bfPsi(\bftheta_t)$ is independent of $\bftheta$, it suffices to bound
\[
\overline{\mc{G}}_t(\bftheta)
:=
\trace(\bfW_t\bfPsi(\bftheta)).
\]

For any $\bftheta,\bftheta' \in \Theta$,
\[
|\overline{\mc{G}}_t(\bftheta \mid \bftheta_t) - \overline{\mc{G}}_t(\bftheta' \mid \bftheta_t)|
=
\left|
\trace
\!\left(
\bfW_t(\bfPsi(\bftheta)-\bfPsi(\bftheta'))
\right)
\right|.
\]

Using the inequality
$
|\trace(\bfA \bfB)|
\le
{m}\,\|\bfA\|_2 \|\bfB\|_2$,
we obtain
\[
|\overline{\mc{G}}_t(\bftheta \mid \bftheta_t) - \overline{\mc{G}}_t(\bftheta' \mid \bftheta_t)|
\le
 {m}\,\|\bfW_t\|_2
\|\bfPsi(\bftheta)-\bfPsi(\bftheta')\|_2.
\]

From the spectral bound, by \Cref{main-ass:psi}, $
\|\bfW_t\|_2
=
\|\bfPsi^{-1}(\bftheta_t)\|_2
\le
\alpha^{-1},$
and by the Lipschitz condition 
\[
\|\bfPsi(\bftheta)-\bfPsi(\bftheta')\|_2
\le
L_\Psi \|\bftheta-\bftheta'\|.
\]
Combining the bounds yields the result.
\end{proof}
Finally, we are ready to prove Lemma~\ref{main-cor:LsmoothnessG}. 
\begin{proof}[Proof of Lemma~\ref{main-cor:LsmoothnessG}]
By definition, we can decompose the surrogate function as:
\begin{align}
    \mc{G}(\bftheta \mid \bftheta_t)= \mc{Q}(\bftheta \mid \bftheta_t) -\log\pi_{\bftheta}(\bftheta)+ \frac{1}{2} \phi(\bftheta),
\end{align}
where $\phi(\bftheta) :=
\bfc^\top(\bftheta)\bfPsi^{-1}(\bftheta)\bfc(\bftheta),$ and $\bfc(\bftheta):= \bfA(\bftheta)\bfmu_\bfx-\bfb.$
To show that $\nabla_1\mc{G}$ is Lipschitz, it is enough to show that the gradients of these three components are Lipschitz continuous.
The prior term follows directly  from the assumptions on $\log \pi_{\bftheta}(\bftheta)$.

\medskip
The gradient of $\mc{Q}$ with respect to $\bftheta$ depends linearly on the Jacobian of $\bfPsi(\bftheta)$. Namely, 
the difference between the $i$-th gradient component at two points, $\bftheta_1$ and $\bftheta_2$:
\begin{align*}
\left| [\nabla_1 \mc{Q}(\bftheta_1 \mid \bftheta_t)]_j - [\nabla_1 \mc{Q}(\bftheta_2 \mid \bftheta_t)]_j \right| =& \left| \trace\!\left[ \bfPsi_t^{-1} \left( \frac{\partial \bfPsi(\bftheta_1)}{\partial \theta_j} - \frac{\partial \bfPsi(\bftheta_2)}{\partial \theta_j} \right) \right] \right|\\
\le & {m} \|\bfPsi_t^{-1}\|_2 \left\| \frac{\partial \bfPsi(\bftheta_1)}{\partial \theta_j} - \frac{\partial \bfPsi(\bftheta_2)}{\partial \theta_j} \right\|_2.
\end{align*}
Because $\nabla \bfPsi(\bftheta)$ is Lipschitz continuous by assumption, and $\|\bfPsi_t^{-1}\|_2\le \alpha^{-1}$. Therefore,  $\nabla_1 \mc{Q}(\bftheta \mid \bftheta_t)$ is Lipschitz continuous.

\medskip
Finally, for the quadratic Term $\phi(\bftheta),$ 
taking the gradient of $\phi(\bftheta)$ using the product rule yields terms involving products of $\bfc(\bftheta)$, its Jacobian $\nabla \bfc(\bftheta)$, the inverse covariance $\bfPsi^{-1}(\bftheta)$, and the derivative of the inverse covariance $\nabla(\bfPsi^{-1}(\bftheta))$. Because $\Theta$ is compact and the components $\bfc(\bftheta)$ and $\bfPsi(\bftheta)$ have Lipschitz continuous gradients, they are continuously differentiable and therefore bounded on any compact set. Thus, there exist universal constants such that $\|\bfc(\bftheta)\|$, $\|\nabla \bfc(\bftheta)\|$, $\|\bfPsi^{-1}(\bftheta)\|_2$, and $\|\nabla \bfPsi(\bftheta)\|_2$ are all strictly bounded for all $\bftheta \in \Theta$.
\end{proof}

\subsection{Sample Complexity for the Stochastic Surrogate}
\label{sec:StochasticSurrogate}
The following theorem establishes the uniform concentration of the stochastic surrogate, providing the basis for the sample complexity $N_t$ required in \Cref{main-thm:coverMMSAA}.

\begin{theorem}[Uniform Concentration {\cite[Theorem 4.3]{saibaba2025stochastic}}]\label{thm:uniform_concentration}
Let \Cref{main-ass:theta,main-ass:psi,main-assump:regularity} hold. At iteration $t$, let 
\[
\bfK(\bftheta) = \frac{1}{2}(\bfPsi^{-1}(\widehat{\bftheta}_t)\bfPsi(\bftheta) + \bfPsi(\bftheta)\bfPsi^{-1}(\widehat{\bftheta}_t))
\]
be the kernel of the trace component of the surrogate $\mc{G}(\cdot \mid \widehat{\bftheta}_t)$. For any $\varepsilon_t > 0$ and $\delta_t \in (0,1)$, let $\gamma_t = S(\frac{\varepsilon_t \alpha}{4 m L_\Psi})$ denote the covering number of $\Theta$. If the number of Rademacher samples $N_t$ satisfies
\begin{equation*}
    N_t \geq \frac{16(2\varsigma_F^2 + \varepsilon_t \varsigma_2)}{\varepsilon_t^2} \ln \left( \frac{2 \gamma_t}{\delta_t} \right),
\end{equation*}
where $\varsigma_F$ and $\varsigma_2$ are defined as in \Cref{main-thm:coverMMSAA}, then the stochastic surrogate satisfies the uniform approximation property
\begin{equation}
    \mathbb{P} \left( \max_{\bftheta \in \Theta} | \mc{G}(\bftheta \mid \widehat{\bftheta}_t) - \widehat{\mc{G}}_t(\bftheta \mid \widehat{\bftheta}_t) | \leq \frac{\varepsilon_t}{2} \right) \geq 1 - \delta_t.
\end{equation}
\end{theorem}

\begin{proof}
To establish the uniform approximation property, we verify the conditions required for the concentration result  \cite[Theorem 4.3]{saibaba2025stochastic}. It is clear that:
\begin{equation*}
|\mc{G}(\bftheta \mid \widehat{\bftheta}_t) - \widehat{\mc{G}}_t(\bftheta \mid \widehat{\bftheta}_t)| = \biggl| \trace(\bfK(\bftheta)) - \frac{1}{N_t} \sum_{i=1}^{N_t} \boldsymbol{\omega}_i^\top \bfPsi^{-1}(\widehat{\bftheta}_t) \bfPsi(\bftheta)\boldsymbol{\omega}_i \biggr|,
\end{equation*}
Under \Cref{main-ass:psi}, the matrix mapping $\bftheta \mapsto \bfPsi(\bftheta)$ is Lipschitz continuous and uniformly positive definite with $\bfPsi(\bftheta) \succeq \alpha \bfI$. It follows from the submultiplicativity of the spectral and Frobenius norms that
\begin{align*}
\|\bfK(\bftheta)\|_\xi \le \|\bfPsi^{-1}(\widehat{\bftheta}_t)\|_2 \|\bfPsi(\bftheta)\|_\xi \le \frac{1}{\alpha} \|\bfPsi(\bftheta)\|_\xi, \quad \xi \in \{2, F\}.
\end{align*}
The analysis in  Theorem 4.3 in \cite{saibaba2025stochastic} also requires a bound on the off-diagonal part,
$\overline{\bfK}(\bftheta) = \bfK(\bftheta) - \diag(\bfK(\bftheta))$. By~\cite{bhatia1989comparing}, 
\begin{align*}
\|\overline{\bfK}(\bftheta)\|_\xi \le   \|\bfK(\bftheta)\|_\xi \le \frac{1}{\alpha} \|\bfPsi(\bftheta)\|_\xi,  \quad \xi \in \{2, F\},
\end{align*}
which yields the constants $\varsigma_F$ and $\varsigma_2$ used in the sample complexity bound.
\end{proof}

\section{Visualization of the majorant}
\label{sec:majorant}
For the seismic inversion example described in \Cref{main-sub:seismic}, we show the majorization of $\mathcal{F}(\bftheta)$.  We consider a small version of the seismic inverse problem where $\bfx$ is a $32\times32$ image so that the full objective function can be computed exactly. We form $\mc{G}(\bftheta \mid\bftheta_t)$ around the point $\bftheta_t=(1.0762\times10^{-4},0.07,0.2)$ using \eqref{main-eq:Gmm} where $\theta_1^2=1.0762\times10^{-4}$ is the true value. The majorant and the full objective function are plotted in \Cref{fig:SImaj} as a function of $\theta_3$ and $\theta_2$ (left), and then as a slice in the $\theta_2$ (middle) and $\theta_3$ (right) directions. 

\begin{figure}
    \centering
    \begin{tabular}{c c c}
        \includegraphics[width=0.3\linewidth]{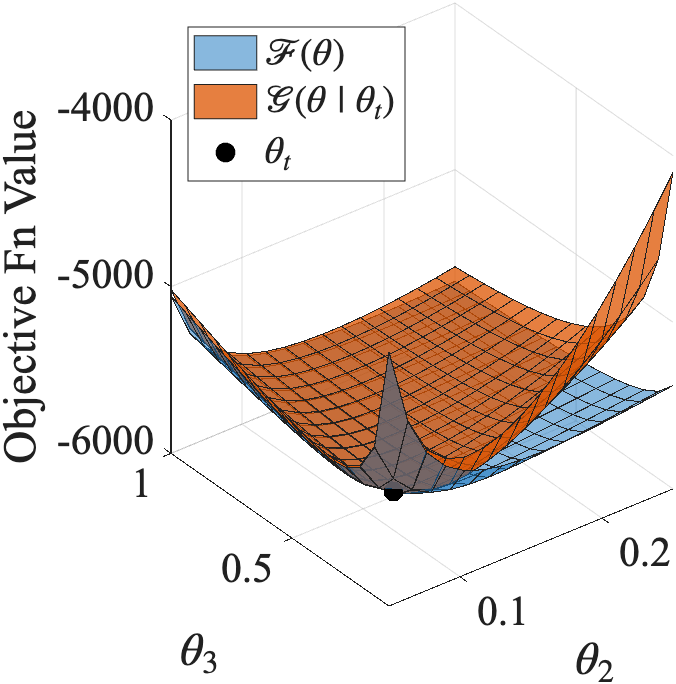} & 
        \includegraphics[width=0.3\linewidth]{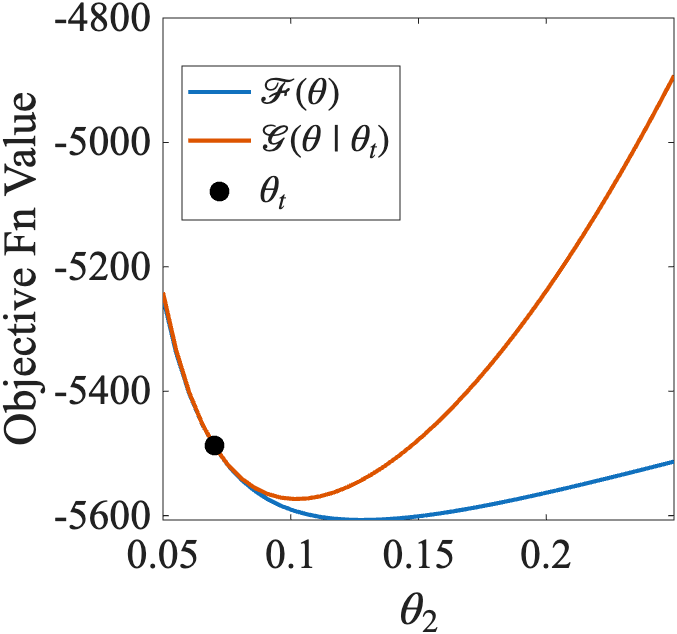} &
        \includegraphics[width=0.3\linewidth]{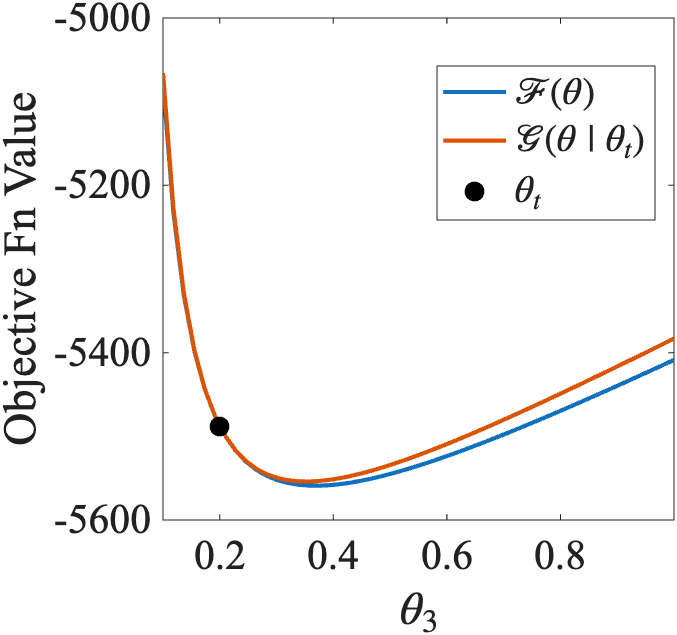} 
    \end{tabular}
    
    \caption{The majorization, $\mc{G}(\bftheta \mid \bftheta_t)$, of $\mc{F}(\bftheta)$ around the point $\bftheta_t=(1.0762\times10^{-4},0.07,0.2)$ for a small seismic inversion problem. }
    \label{fig:SImaj}
\end{figure}

\bibliographystyle{siamplain}
\bibliography{references}